\documentclass[final,3p]{elsarticle}
 \usepackage[latin1]{inputenc}
 \usepackage{graphics}
 \usepackage{graphicx}
 \usepackage{epsfig}
\usepackage{amssymb}
 \usepackage{amsthm}
 \usepackage{lineno}
 \usepackage{amsmath}
   \numberwithin{equation}{section}
\usepackage{mathrsfs}

\newtheorem{thm}{Theorem}[section]

\newtheorem{lem}[thm]{Lemma}

\newtheorem{defn}[thm]{Definition}

\biboptions{sort&compress,square}
\begin{document}
\begin{frontmatter}
\author[rvt1]{Jian Wang}
\ead{wangj@tute.edu.cn}
\author[rvt2]{Yong Wang\corref{cor2}}
\ead{wangy581@nenu.edu.cn}

\cortext[cor2]{Corresponding author.}
\address[rvt1]{School of Science, Tianjin University of Technology and Education, Tianjin, 300222, P.R.China}
\address[rvt2]{School of Mathematics and Statistics, Northeast Normal University,
Changchun, 130024, P.R.China}

\title{ Equivariant Bismut Laplacian and  spectral Einstein functional}
\begin{abstract}

This paper aims to  provide an explicit computation of the equivariant noncommutative residue density
of which yield the metric and Einstein tensors  on even-dimensional Riemannian manifolds.
A considerable contribution of this paper is the development of
the spectral Einstein functionals by two vector fields and the equivariant Bismut Laplacian over spinor bundles.
We prove the equivariant Dabrowski-Sitarz-Zalecki type
  theorems for lower dimensional spin manifolds   with (or without) boundary.
\end{abstract}
\end{frontmatter}
\section{Introduction}
\label{1}
In Connes¡¯ program of noncommutative geometry,
for associative unital algebra $\mathcal{A}$ and Hilbert space $\mathcal{H}$ such that there is an algebra
 homomorphism $\pi:\mathcal{A}\rightarrow B(\mathcal{H})$, where $B(\mathcal{H})$
denotes the algebra of bounded operators acting on $\mathcal{H}$,
the role of geometrical
objects is played by spectral triples $(\mathcal{A}; \mathcal{H}; D )$.
Similar to the commutative case and
the canonical spectral triple $(C^{\infty}(M); L^{2}(S); D )$, where $(M; g; S )$ is a closed spin
manifold and $D$  is the Dirac operator acting on the spinor bundle $S$, the spectrum
of the Dirac operator $D$  of a spectral triple $(\mathcal{A}; \mathcal{H}; D )$ encodes the geometrical
information of the spectral triple.
An important new feature of such geometries, which is absent in the
commutative case, is the existence of inner fluctuations of the metric.
No doubt it would be extremely interesting to recover scalar curvature, Ricci curvature and other important tensors in both the
 classical setup as well as for the noncommutative geometry or quantum geometry.
The spectral-theoretic approach to scalar curvature has been extended  to quantum tori in the seminal work of
 Connes and Tretkoff \cite{CoT}, expanded in \cite{CoM} and then extensively studied by many authors \cite{DaS,Fa,Ka,KW}.
 Unlike the scalar curvature, the Ricci curvature does not appear in the coefficients
of the heat trace of the Dirac Laplacian. Floricel etc. \cite{FGK} observed that
the Laplacian of the de Rham complex, more precisely the Laplacian on one forms, captures the Ricci
operator in its second term, and formulated the Ricci
operator as a spectral functional on the algebra of sections of the endomorphism
bundle of the cotangent bundle of $M$:
\begin{align}
Ric (F)=a_{2}(tr(F), \triangle_{0})-a_{2}(F, \triangle_{1}),~~~F\in C^{\infty}(\mathrm{End}(T^{*}M)).
\end{align}
Recently, for the metric tensor $g$, Ricci curvature $Ric$ and the scalar curvature $s$,  Dabrowski etc. \cite{DL} recovered the Einstein tensor
\begin{equation}
G:=Ric-\frac{1}{2}s(g)g,
\end{equation}
which directly enters the Einstein field equations with matter, and its dual.

The method of producing these spectral functionals in \cite{DL} is the Wodzicki residue (or be written as  noncommutative residue) density
of a suitable power of the
 Laplace type operator multiplied by a pair of other differential operators.
Let $E$ be a finite-dimensional complex vector bundle over a closed compact manifold $M$
of dimension $n$, the noncommutative residue of a pseudo-differential operator
$P\in\Psi DO(E)$ can be defined by
 \begin{equation}
res(P):=(2\pi)^{-n}\int_{S^{*}M}\mathrm{Tr}(\sigma_{-n}^{P}(x,\xi))\mathrm{d}x \mathrm{d}\xi,
\end{equation}
where $S^{*}M\subset T^{*}M$ denotes the co-sphere bundle on $M$ and
$\sigma_{-n}^{P}$ is the component of order $-n$ of the complete symbol
$
\sigma^{P}:=\sum_{i}\sigma_{i}^{P}
$
of $P$, cf. \cite{Ac,Wo,Wo1,Gu},
and the linear functional $res: \Psi DO(E)\rightarrow \mathbb{C }$
is in fact the unique trace (up to multiplication
by constants) on the algebra of pseudo-differential operators $\Psi DO(E)$.
Connes  proved that the noncommutative residue on a compact manifold $M$ coincided with Dixmier's trace on pseudodifferential
operators of order -dim$M$\cite{Co2}. More precisely, Connes made a challenging observation that
the Wodzicki residue ${\rm  Wres}(D^{-2})$ is proportional to the Einstein-Hilbert action of general relativity
 (which, in turn is essentially the $a_{2}$ coefficient of the heat expansion of $\Delta$), that is called the Kastler-Kalau-Walze theorem now.
The earlier proofs of Kastler \cite{Ka} and of Kalau, Walze \cite{KW} are based on the explicit computation of
$\sigma_{-4}{(D^{-2})}$, which is most easily performed in normal coordinates (this is the main advantage of the KW approach
over the ¡®brute force¡¯ computation by Kastler).
Fedosov etc.\cite{FGLS} constructed a noncommutative
residue on the algebra of classical elements in Boutet de Monvel's calculus on a compact manifold with boundary of dimension $n>2$.
For Dirac operators and signature operators, Wang computed the noncommutative residue and proved Kastler-Kalau-Walze type theorem
for manifolds with boundary\cite{Wa1,Wa3,Wa4}.

Notably, the equivariant Riemannian curvature has the form
 \begin{equation}
R_{g}(X)=(\nabla-\iota(X))^{2}+\mathcal{L}_{X},
\end{equation}
where $\iota(X)$ denotes the contraction operator $\iota(X_{M})$, $\mathcal{L}_{X}=\nabla_{X}+ \mu (X)$ and
the moment map given by  $\mu (X)Y=-\nabla_{Y}X$.
Earlier Jean-Michel Bismut proved the local infinitesimal equivariant index theorem for equivariant
Laplacian  associated with group action \cite{JMB}, and the Bismut Laplacian  given by the formula
 \begin{equation}
H(X)=(D+\frac{1}{4}c(X))^{2}+\mathcal{L}_{X},
\end{equation}
which may be thought of as a quantum analogue of the equivariant Riemannian curvature.
Let $T$ be a finite group acting on a compact manifold $M$ and   $\mathcal{A}(M)$ denote the
algebra of classical complete symbols on $M$,
Dave \cite{SDa} determined all traces on the cross-product
algebra $\mathcal{A}(M)\bowtie T$ as residues of certain meromorphic zeta functions
by a special technique.
By computing the homologies of $\mathcal{A}(M)\bowtie T$, Dave obtained
an equivariant Connes trace formula as well as extensions of the logarithmic symbols
based on a two cocycle in $H^{2}(S_{log}(M)\times T)$ \cite{KV},
which as equivariant noncommutative residue associated to the conjugacy class.
In \cite{Wa5}, Wang proved equivariant Kastler-Kalau-Walze theorems for spin manifolds without boundary
and six dimensional spin manifolds with boundary.
Motivated by the spectral Einstein functionals, equivariant Bismut Laplacian and  noncommutative residue,
  we give some new equivariant Einstein  functionals which is the extension of spectral Einstein functionals to the equivariant noncommutative
   realm,  and we relate them to  the noncommutative residue for lower dimensional manifolds with boundary.

\section{Spectral Einstein functionals for equivariant Bismut Laplacian}
The purpose of this section is to work out what are the spectral Einstein functionals needed to
specify the equivariant Bismut Laplacian. Let $M$ be a compact oriented Riemannian manifold of even dimension $n$, and let $L$ be a compact Lie group
  acting on $M$. Let  $\mathcal{E}$ be a $L-$equivariant Clifford module over $M$ with
  $L-$equivariant Hermitian structure and   $L-$invariant Clifford connection $\nabla^{\mathcal{E}}$.
\begin{defn}
A Dirac operator $D$ on a $\mathbb{Z}_{2}$-graded vector bundle $\mathcal{E}$ is a
first-order differential operator of odd parity on $\mathcal{E}$,
 \begin{align}
D:\Gamma(M,\mathcal{E}^{\pm})\rightarrow\Gamma(M,\mathcal{E}^{\mp}),
\end{align}
such that $D^{2}$ is a generalized Laplacian.
\end{defn}
\begin{defn}\cite{BGV}
The Bismut Laplacian  is the second-order differential operator on $ \Gamma(M,\mathcal{E})$  given by the formula
 \begin{align}
H(X)=\big(D+\frac{1}{4}c(X)\big)^{2}+\mathcal{L}^{S(TM)}_{X}
=\big(D+\frac{1}{4}c(X)\big)^{2}+\nabla^{S(TM)}_{X}+ \mu^{S(TM)}(X),
\end{align}
where $\mu^{S(TM)}(X)=\mathcal{L}^{S(TM)}_{X}- \nabla^{S(TM)}_{X}$, $c(X)$ denotes the Clifford action of the dual
of the vector field on $\mathcal{E}$, which satisfies the relation
 \begin{align}
 c(e_{i})c(e_{j})+c(e_{j})c(e_{i})=-2\langle e_{i}, e_{j}  \rangle.
\end{align}
\end{defn}
 Let $\nabla^L$ denote the Levi-Civita connection about $g^M$. In the
 fixed orthonormal frame $\{e_1,\cdots,e_n\}$ and natural frames $\{\partial_1,\cdots,\partial_n\}$ of $TM$,
the connection matrix $(\omega_{s,t})$ defined by
\begin{equation}
\nabla^L(e_1,\cdots,e_n)= (e_1,\cdots,e_n)(\omega_{s,t}).
\end{equation}
 Then the Dirac operator has the form
\begin{equation}
D=\sum^n_{i=1}c(e_i)\nabla_{e_i}^{S}=\sum^n_{i=1}c(e_i)\Big[e_i
-\frac{1}{4}\sum_{s,t}\omega_{s,t}(e_i)c(e_s)c(e_t)\Big].
\end{equation}
In what follows, using the Einstein sum convention for repeated index summation:
$\partial^j=g^{ij}\partial_i$,$\sigma^i=g^{ij}\sigma_{j}$, $\Gamma_{k}=g^{ij}\Gamma_{ij}^{k}$,
 and we omit the summation sign, then
\begin{equation}
\nabla_{\partial_i}^{S}=\partial_i+\sigma_i=\partial_i-\frac{1}{4}\sum_{s,t}\omega_{s,t}(e_i)c(e_s)c(e_t) .
\end{equation}
By (2.11) in \cite{Wa5}, for $X=\sum_{j=1}^{n}X_{j}\partial_{j} $, we have
\begin{align}
H_{X}
 =&D^{2}+\frac{1}{4}Dc(X)+\frac{1}{4}c(X)D+\nabla_{ X}^{S}+\mu(X)-\frac{1}{16}|X|^{2}\nonumber\\
  =&-g^{ij}\partial_i\partial_j+\Big(X_{j}-2\sigma^{j}+\Gamma^{j}+\frac{1}{4}c(\partial^j)c(X)+\frac{1}{4}c(X)c(\partial^j)\Big)\partial_j\nonumber\\
  &+g^{ij}\Big(-\partial_i(\sigma_{j})-\sigma_{i}\sigma_{j}+\Gamma_{ij}^{k}\sigma_{k}
  +\frac{1}{4}c(\partial_i)\partial_i(c(X))+\frac{1}{4}c(\partial_i) \sigma_{j}c(X)\nonumber\\
  &+\frac{1}{4}c(X)c(\partial_i) \sigma_{j}\Big)+\frac{1}{4}s-\frac{1}{16}|X|^{2}+\frac{1}{4}X_{j}\sigma_{j}+\mu(X).
  \end{align}

The following lemma of Dabrowski etc.'s Einstein functional play a key role in our proof
of the Einstein functional for equivariant Bismut Laplacian.
Let $V$, $W$ be a pair of vector fields on a compact
Riemannian manifold $M$, of dimension $n = 2m$. Using the equivariant Bismut Laplacian $H_{X}^{-1} =\Delta+E $
acting on sections of
 a vector bundle $S(TM)$ of rank $2^{m}$ which
 may contain both some nontrivial connections,
 the spectral functionals over vector fields defined by
\begin{lem}\cite{DL}
The spectral  Einstein functional equal to
 \begin{equation}
Wres\big(\widetilde{\nabla}_{V}\widetilde{\nabla}_{W}H_{X}^{-m} \big)=\frac{\upsilon_{n-1}}{6}2^{m}\int_{M}G(V,W)vol_{g}
 +\frac{\upsilon_{n-1}}{2}\int_{M}F(V,W)vol_{g}+\frac{1}{2}\int_{M}(\mathrm{Tr}E)g(V,W)vol_{g},
\end{equation}
where $G(V,W)$ denotes the Einstein tensor evaluated on the two vector fields, $F(V,W)=Tr(V_{a}W_{b}F_{ab})$ and
$F_{ab}$ is the curvature tensor of the connection $\widetilde{\nabla}$, $\mathrm{Tr}E$ denotes the trace of $E$ and $\upsilon_{n-1}=\frac{2\pi^{m}}{\Gamma(m)}$.
\end{lem}
The aim of this section is to prove the following.
\begin{thm}
For the equivariant Bismut Laplacian $H_{X}$, the spectral  Einstein functional equal to
\begin{align}
&Wres\Big(\big(\nabla^{S(TM)}_{V}+\frac{3}{4}g(V,X)\big)\big(\nabla^{S(TM)}_{W}+\frac{3}{4}g(W,X)\big)H_{X}^{-m}\Big)\nonumber\\
=&\frac{\upsilon_{n-1}}{6}2^{m}\int_{M}\big(Ric(V,W)-\frac{1}{2}sg(V,W)\big) vol_{g}
+\frac{3\upsilon_{n-1}}{8}2^{m}\int_{M} \Big(g\big(\nabla_{V}^{TM}X,W\big)-g\big(\nabla_{W}^{TM}X,V\big) \Big)vol_{g}\nonumber\\
&+2^{m-1}\int_{M}\Big(\frac{1}{4}s^{g}+\frac{1}{2}\mathrm{div}^{g}(X)+\frac{1}{2}|X|^{2}
-Tr\big( \mu^{S(TM)}(X)\big) \Big)g(V,W) vol_{g},
\end{align}
where $\mathrm{div}^{g}(X)$ denotes the divergence of $X$, and $s$ denotes the scalar curvature.
\end{thm}

\begin{proof}
First of all, with the same assumptions as  lemma 2.3, we must study $F$ and $E$ in the above coordinate system.
Let $V$ be a vector bundle on $M$, any differential operator $P$ of Laplace type has locally the form
 \begin{equation}
P=-\big(g^{ij}\partial_{i}\partial_{j}+A^{i}\partial_{i}+B\big),
\end{equation}
where $\partial_{i}$  is a natural local frame on $TM$ ,
 and $(g^{ij})_{1\leq i,j\leq n}$ is the inverse matrix associated with the metric
matrix  $(g_{ij})_{1\leq i,j\leq n}$ on $M$,
 and $A^{i}$ and $B$ are smooth sections
of $\texttt{End}(V)$ on $M$ \cite{PBG}. If $P$ is a Laplace type
operator of the form (2.10), then there is a unique
connection $\nabla$ on $V$ and a unique Endomorphism $E$ such that
 \begin{equation}
P=-\Big[g^{ij}(\nabla_{\partial_{i}}\nabla_{\partial_{j}}-
 \nabla_{\nabla^{L}_{\partial_{i}}\partial_{j}})+E\Big],
\end{equation}
where $\nabla^{L}$ denotes the Levi-Civita connection on $M$. Moreover
(with local frames of $T^{*}M$ and $V$), $\nabla_{\partial_{i}}=\partial_{i}+\omega_{i} $
and $E$ are related to $g^{ij}$, $A^{i}$, and $B$ through
 \begin{eqnarray}
&&\omega_{i}=\frac{1}{2}g_{ij}\big(A^{i}+g^{kl}\Gamma_{ kl}^{j} \texttt{Id}\big),\\
&&E=B-g^{ij}\big(\partial_{i}(\omega_{i})+\omega_{i}\omega_{j}-\omega_{k}\Gamma_{ ij}^{k} \big),
\end{eqnarray}
where $\Gamma_{ kl}^{j}$ is the  Christoffel coefficient of $\nabla^{L}$,
and the above $E$ equals $(-E)$ in \cite{DL}.

Let
 \begin{align}
&\big(D+\frac{1}{4}c(X)\big)^{2}=-\Big( g^{ij}(\nabla_{\partial_{i}}\nabla_{\partial_{i}}-\nabla_{\nabla^{L}_{\partial_{i}}\partial_{j}} )+E  \Big),
\nonumber\\
&H(X)=\big(D+\frac{1}{4}c(X)\big)^{2}+\nabla_{X}+ \mu (X)=
-\Big( g^{ij}(\widetilde{\nabla}_{\partial_{i}}\widetilde{\nabla}_{\partial_{i}}-\widetilde{\nabla}_{\nabla^{L}_{\partial_{i}}\partial_{j}} )
+\widetilde{E}  \Big).
\end{align}
Then we obtain
 \begin{align}
\widetilde{A}_{j}=A_{j}+X_{j},\widetilde{B} =B+\frac{1}{4}X_{j}\sigma_{j}+\mu(X),
\end{align}
and
 \begin{align}
\widetilde{\omega}_{i}=&\omega_{i}+\frac{1}{2}g_{ij}X_{j}
= \sigma_{i}-\frac{1}{8} c(\partial_{i})c(X)-\frac{1}{8}c(X) c(\partial_{i})+\frac{1}{2}g_{ij}X_{j}.
\end{align}
Substituting (2.15),(2.16) into (2.13) yield
 \begin{align}
&E=-\frac{1}{4}s+\frac{1}{16}|X|^{2}+\frac{1}{2}\Big[e_{j}\Big(\frac{1}{4}c(X)\Big)c(e_{j})-c(e_{j})e_{j}\Big(\frac{1}{4}c(X)\Big)
\Big]- \frac{1}{64}\Big(c(e_{j})c(X)+c(X)c(e_{j})\Big)^{2},\nonumber\\
&\widetilde{E}=E+ \frac{1}{4}X_{j}\sigma_{j}+\mu(X)-g^{ij}\Big[\partial_{i}\big(  \frac{1}{2}g_{jk}X_{k}\big)
+\frac{1}{2}\omega_{i}g_{ik}X_{k}+\frac{1}{2}g_{il}X_{l}\omega_{j}+\frac{1}{4}g_{ik}g_{jl}X_{k}X_{l}-\frac{1}{2}g_{kl}X_{l}\Gamma_{ij}^{k}\Big].
\end{align}
From (2.6) and (2.16) we obtain
\begin{equation}
\widetilde{\nabla}_{V}=\nabla^{S}_{V}+\frac{1}{4}\langle V, X  \rangle+\frac{1}{2}g(V,X)=\nabla
^{S}_{V}+\frac{3}{4}g(V,X).
\end{equation}
Then  we take the normal coordinate about $x_0$, it follows that
\begin{align}
\mathrm{Tr}^{S(TM)}(E)(x_{0})=&\mathrm{Tr}^{S(TM)}\Big[-\frac{1}{4}s+\frac{1}{16}|X|^{2}
-\frac{1}{16}|X|^{2}\nonumber\\
&+\frac{1}{2}\Big(e_{j}\big(\frac{1}{4}c(X)\big)c(e_{j})-c(e_{j})e_{j}\big(\frac{1}{4}c(X)\big)
\Big)
\Big](x_{0})\nonumber\\
=&\mathrm{Tr}^{S(TM)}\Big(-\frac{1}{4}s+\frac{1}{16}|X|^{2}-\frac{1}{16}|X|^{2}\Big)(x_{0})\nonumber\\
=&-\frac{1}{4}s \mathrm{Tr}^{S(TM)} (\mathrm{Id} )(x_{0}),
\end{align}
where we note
\begin{align}
\mathrm{Tr}^{S(TM)} \Big[e_{j}\Big(\frac{1}{4}c(X)\Big)c(e_{j})-c(e_{j})e_{j}\Big(\frac{1}{4}c(X)\Big)
\Big](x_{0})=0.
\end{align}
Now, what is left is to show that
 \begin{align}
-g^{ij} \partial_{i}\big(  \frac{1}{2}g_{jk}X_{k}\big)(x_{0})
=&-\frac{1}{2}\partial_{j}\big( g_{jk}X_{k}\big)(x_{0})\nonumber\\
=&-\frac{1}{2}\Big(\partial_{j}\big( g_{jk}\big)(x_{0})X_{k}+ g_{jk}(x_{0})\partial_{j}\big(X_{k}\big)(x_{0})   \Big)\nonumber\\
=&-\frac{1}{2}\partial_{j}(X_{j})(x_{0})=-\frac{1}{2}\mathrm{div}(X),
\end{align}
and
 \begin{align}
\sum_{j=1}^{n}\langle \nabla_{ e_{j}}^{TM}X, e_{j} \rangle(x_{0})
=&\sum_{j,k=1}^{n}\langle \nabla_{ e_{j}}^{TM}(X_{k}\partial_{k}), e_{j} \rangle(x_{0})\nonumber\\
=&\sum_{j,k=1}^{n}\langle e_{j}(X_{k})\partial_{k}+X_{k}\nabla_{ e_{j}}^{TM}(\partial_{k}), e_{j} \rangle(x_{0})\nonumber\\
=&\sum_{j=1}^{n}\partial_{j}(X_{j})(x_{0})= \mathrm{div}(X).
\end{align}
An easy calculation gives
 \begin{align}
-\frac{1}{2}g^{ij}\omega_{i}g_{ik}X_{k}(x_{0})=&-\frac{1}{2}\sum_{k}\omega_{k}X_{k}(x_{0})\nonumber\\
=&-\frac{1}{2}\sum_{k}\Big( -\frac{1}{8}c(\partial_{k})c(X)-\frac{1}{8}c(X) c(\partial_{k}) \Big)X_{k}(x_{0})\nonumber\\
=&\frac{1}{16}\Big(c^{ 2}(X)+c^{ 2}(X)  \Big)(x_{0})=\frac{1}{8}c^{ 2}(X)=-\frac{1}{8}|X|^{2},
\end{align}
and
 \begin{align}
-\frac{1}{2}g^{ij}g_{il}X_{l}\omega_{j}(x_{0})&=-\frac{1}{2}\sum_{l}\omega_{l}X_{l}(x_{0})=-\frac{1}{8}|X|^{2}.
\end{align}
Similarly,
 \begin{align}
-\frac{1}{4}g^{ij}g_{ik}g_{jl}X_{k}X_{l}(x_{0}) =&-\sum_{i,j,k,l}\frac{1}{4}g^{ij}g_{ik}g_{jl}X_{k}X_{l}(x_{0})\nonumber\\
=&-\sum_{ j,k,l}\frac{1}{4}\delta^{jk}g_{jl}X_{k}X_{l}(x_{0})\nonumber\\
=&-\frac{1}{4}\sum_{ j }X_{j}^{2}(x_{0})=-\frac{1}{4}|X|^{2}.
\end{align}
Summing up (2.18)-(2.25) leads to the desired equality
\begin{align}
\mathrm{Tr}^{S(TM)}(\widetilde{E})(x_{0})= \Big(-\frac{1}{4}s^{g}-\frac{1}{2}\mathrm{div}^{g}(X)
-\frac{1}{2}|X|^{2}\Big)\mathrm{Tr}^{S(TM)} (\mathrm{Id} )(x_{0})
+ Tr(\mu^{S(TM)}(X)(x_{0}).
\end{align}

On the other hand, let $V=\sum_{a=1}^{n}V^{a}e_{a}$,$W=\sum_{b=1}^{n}W^{b}e_{b}$,
we have
 \begin{equation}
F(V,W)=Tr(V_{a}W_{b}F_{ab})=\sum_{a,b=1}^{n}V^{a}W^{b}Tr^{S(TM)}(F_{e_{a},e_{b}}).
\end{equation}
Let
\begin{equation}
\widetilde{\nabla}_{e_{a}} =\nabla^{S}_{e_{a}}+\frac{3}{4}g(X,e_{a})=:e_{a}+\sigma_{a}+\frac{3}{4}X_{a},
\end{equation}
then
 \begin{align}
F_{e_{a},e_{b}}(x_{0})=&e_{a}\Big(\sigma_{b}(x_{0})+\frac{3}{4}X_{b} \Big)(x_{0})-e_{b}\Big(\sigma_{a}+\frac{3}{4}X_{a} \Big)(x_{0})\nonumber\\
 &+ \Big(\sigma_{a}+\frac{3}{4}X_{a} \Big) \Big(\sigma_{b}+\frac{3}{4}X_{b} \Big)(x_{0})
 - \Big(\sigma_{b}+\frac{3}{4}X_{b} \Big) \Big(\sigma_{a}+\frac{3}{4}X_{a} \Big)(x_{0})\nonumber\\
 &- \Big(\sigma([e_{a},e_{b}])+\frac{3}{4}\langle X , [e_{a},e_{b}] \rangle \Big)(x_{0}).
\end{align}
It follows that
 \begin{align}
Tr^{S(TM)}\big(e_{a}\sigma_{b}\big)(x_{0})=0,Tr^{S(TM)}\big(e_{b}\sigma_{b}\big)(x_{0})=0,\sigma([e_{a},e_{b}])(x_{0})=0,
\end{align}
and
 \begin{align}
&V(g(X,W))-W(g(X,V))-\langle X , [V,W] \rangle\nonumber\\
 =&V\Big(\sum_{b}X_{b}W_{b}\Big)-W\Big(\sum_{a}X_{a}V_{a}\Big)-\Big\langle X , \big[\sum_{a}V_{a}e_{a},\sum_{b}W_{b}e_{b}\big] \Big\rangle \nonumber\\
 =&\sum_{b}V(X_{b})W_{b}+\sum_{b}X_{b}V(W_{b})-\sum_{a}W(X_{a})V_{a}-\sum_{a}X_{a}W(V_{a})\nonumber\\
   &-\Big\langle X , \sum_{a,b}V_{a}W_{b}[e_{a},e_{b}]\Big\rangle-\Big\langle X , \sum_{b}V(W_{b})e_{b} \Big\rangle\nonumber\\
   &+\Big\langle X , \sum_{ab}W(V_{a})e_{a} \Big\rangle.
\end{align}
Then we obtain
 \begin{align}
F(V,W)=&Tr(V_{a}W_{b}F_{ab})(x_{0})=\sum_{a,b=1}^{n}V^{a}W^{b}Tr^{S(TM)}(F_{e_{a},e_{b}})(x_{0})\nonumber\\
=&\frac{3}{4}\Big(g\big(\nabla_{V}^{TM}X,W\big)-g\big(\nabla_{W}^{TM}X,V\big) \Big)(x_{0})\mathrm{Tr}^{S(TM)} (\mathrm{Id} ).
\end{align}
Summing up (2.26), (2.32) leads to the desired equality (2.9), and the proof of the Theorem is complete.
\end{proof}

\section{Spectral Einstein functionals of the Bismut Laplacian for manifolds with boundary}

\subsection{Noncommutative residue of the equivariant Bismut Laplacian for manifolds with boundary}
In this section, for compact oriented manifold $M$ with boundary $\partial M$,
to define the noncommutative residue associated with equivariant Bismut Laplacian,
some basic facts and formulae about Boutet de Monvel's calculus are recalled as follows.

Let $$ F:L^2({\bf R}_t)\rightarrow L^2({\bf R}_v);~F(u)(v)=\int e^{-ivt}u(t)\texttt{d}t$$ denote the Fourier transformation and
$\varphi(\overline{{\bf R}^+}) =r^+\varphi({\bf R})$ (similarly define $\varphi(\overline{{\bf R}^-}$)), where $\varphi({\bf R})$
denotes the Schwartz space and
  \begin{equation}
r^{+}:C^\infty ({\bf R})\rightarrow C^\infty (\overline{{\bf R}^+});~ f\rightarrow f|\overline{{\bf R}^+};~
 \overline{{\bf R}^+}=\{x\geq0;x\in {\bf R}\}.
\end{equation}
We define $H^+=F(\varphi(\overline{{\bf R}^+}));~ H^-_0=F(\varphi(\overline{{\bf R}^-}))$ which are orthogonal to each other. We have the following
 property: $h\in H^+~(H^-_0)$ iff $h\in C^\infty({\bf R})$ which has an analytic extension to the lower (upper) complex
half-plane $\{{\rm Im}\xi<0\}~(\{{\rm Im}\xi>0\})$ such that for all nonnegative integer $l$,
 \begin{equation}
\frac{\texttt{d}^{l}h}{\texttt{d}\xi^l}(\xi)\sim\sum^{\infty}_{k=1}\frac{\texttt{d}^l}{\texttt{d}\xi^l}(\frac{c_k}{\xi^k})
\end{equation}
as $|\xi|\rightarrow +\infty,{\rm Im}\xi\leq0~({\rm Im}\xi\geq0)$.

 Let $H'$ be the space of all polynomials and $H^-=H^-_0\bigoplus H';~H=H^+\bigoplus H^-.$ Denote by $\pi^+~(\pi^-)$ respectively the
 projection on $H^+~(H^-)$. For calculations, we take $H=\widetilde H=\{$rational functions having no poles on the real axis$\}$ ($\tilde{H}$
 is a dense set in the topology of $H$). Then on $\tilde{H}$,
 \begin{equation}
\pi^+h(\xi_0)=\frac{1}{2\pi i}\lim_{u\rightarrow 0^{-}}\int_{\Gamma^+}\frac{h(\xi)}{\xi_0+iu-\xi}\texttt{d}\xi,
\end{equation}
where $\Gamma^+$ is a Jordan close curve included ${\rm Im}\xi>0$ surrounding all the singularities of $h$ in the upper half-plane and
$\xi_0\in {\bf R}$. Similarly, define $\pi^{'}$ on $\tilde{H}$,
 \begin{equation}
\pi'h=\frac{1}{2\pi}\int_{\Gamma^+}h(\xi)\texttt{d}\xi.
\end{equation}
So, $\pi'(H^-)=0$. For $h\in H\bigcap L^1(R)$, $\pi'h=\frac{1}{2\pi}\int_{R}h(v)\texttt{d}v$ and for $h\in H^+\bigcap L^1(R)$, $\pi'h=0$.
Denote by $\mathcal{B}$ Boutet de Monvel's algebra (for details, see Section 2 of \cite{Wa1}).

An operator of order $m\in {\bf Z}$ and type $d$ is a matrix
$$A=\left(\begin{array}{lcr}
  \pi^+P+G  & K  \\
   T  &  S
\end{array}\right):
\begin{array}{cc}
\   C^{\infty}(X,E_1)\\
 \   \bigoplus\\
 \   C^{\infty}(\partial{X},F_1)
\end{array}
\longrightarrow
\begin{array}{cc}
\   C^{\infty}(X,E_2)\\
\   \bigoplus\\
 \   C^{\infty}(\partial{X},F_2)
\end{array}.
$$
where $X$ is a manifold with boundary $\partial X$ and
$E_1,E_2~(F_1,F_2)$ are vector bundles over $X~(\partial X
)$.~Here,~$P:C^{\infty}_0(\Omega,\overline {E_1})\rightarrow
C^{\infty}(\Omega,\overline {E_2})$ is a classical
pseudodifferential operator of order $m$ on $\Omega$, where
$\Omega$ is an open neighborhood of $X$ and
$\overline{E_i}|X=E_i~(i=1,2)$. $P$ has an extension:
$~{\cal{E'}}(\Omega,\overline {E_1})\rightarrow
{\cal{D'}}(\Omega,\overline {E_2})$, where
${\cal{E'}}(\Omega,\overline {E_1})~({\cal{D'}}(\Omega,\overline
{E_2}))$ is the dual space of $C^{\infty}(\Omega,\overline
{E_1})~(C^{\infty}_0(\Omega,\overline {E_2}))$. Let
$e^+:C^{\infty}(X,{E_1})\rightarrow{\cal{E'}}(\Omega,\overline
{E_1})$ denote extension by zero from $X$ to $\Omega$ and
$r^+:{\cal{D'}}(\Omega,\overline{E_2})\rightarrow
{\cal{D'}}(\Omega, {E_2})$ denote the restriction from $\Omega$ to
$X$, then define
$$\pi^+P=r^+Pe^+:C^{\infty}(X,{E_1})\rightarrow {\cal{D'}}(\Omega,
{E_2}).$$
In addition, $P$ is supposed to have the
transmission property; this means that, for all $j,k,\alpha$, the
homogeneous component $p_j$ of order $j$ in the asymptotic
expansion of the
symbol $p$ of $P$ in local coordinates near the boundary satisfies:
$$\partial^k_{x_n}\partial^\alpha_{\xi'}p_j(x',0,0,+1)=
(-1)^{j-|\alpha|}\partial^k_{x_n}\partial^\alpha_{\xi'}p_j(x',0,0,-1),$$
then $\pi^+P:C^{\infty}(X,{E_1})\rightarrow C^{\infty}(X,{E_2})$
by Section 2.1 of \cite{Wa1}.

Let $M$ be a compact manifold with boundary $\partial M$. We assume that the metric $g^{M}$ on $M$ has
the following form near the boundary
 \begin{equation}
 g^{M}=\frac{1}{h(x_{n})}g^{\partial M}+\texttt{d}x _{n}^{2} ,
\end{equation}
where $g^{\partial M}$ is the metric on $\partial M$. Let $U\subset
M$ be a collar neighborhood of $\partial M$ which is diffeomorphic $\partial M\times [0,1)$. By the definition of $h(x_n)\in C^{\infty}([0,1))$
and $h(x_n)>0$, there exists $\tilde{h}\in C^{\infty}((-\varepsilon,1))$ such that $\tilde{h}|_{[0,1)}=h$ and $\tilde{h}>0$ for some
sufficiently small $\varepsilon>0$. Then there exists a metric $\hat{g}$ on $\hat{M}=M\bigcup_{\partial M}\partial M\times
(-\varepsilon,0]$ which has the form on $U\bigcup_{\partial M}\partial M\times (-\varepsilon,0 ]$
 \begin{equation}
\hat{g}=\frac{1}{\tilde{h}(x_{n})}g^{\partial M}+\texttt{d}x _{n}^{2} ,
\end{equation}
such that $\hat{g}|_{M}=g$.
We fix a metric $\hat{g}$ on the $\hat{M}$ such that $\hat{g}|_{M}=g$.
Now we recall the main theorem in \cite{FGLS}.
\begin{thm}\label{th:32}{\bf(Fedosov-Golse-Leichtnam-Schrohe)}
 Let $X$ and $\partial X$ be connected, ${\rm dim}X=n\geq3$,
 $A=\left(\begin{array}{lcr}\pi^+P+G &   K \\
T &  S    \end{array}\right)$ $\in \mathcal{B}$ , and denote by $p$, $b$ and $s$ the local symbols of $P,G$ and $S$ respectively.
 Define:
 \begin{align}
{\rm{\widetilde{Wres}}}(A)=&\int_X\int_{\bf S}{\mathrm{Tr}}_E\left[p_{-n}(x,\xi)\right]\sigma(\xi)dx \nonumber\\
&+2\pi\int_ {\partial X}\int_{\bf S'}\left\{{\mathrm{Tr}}_E\left[({\mathrm{Tr}}b_{-n})(x',\xi')\right]+{\mathrm{Tr}}
_F\left[s_{1-n}(x',\xi')\right]\right\}\sigma(\xi')dx',
\end{align}
Then~~ a) ${\rm \widetilde{Wres}}([A,B])=0 $, for any
$A,B\in\mathcal{B}$;~~ b) It is a unique continuous trace on
$\mathcal{B}/\mathcal{B}^{-\infty}$.
\end{thm}

 Let $p_{1},p_{2}$ be nonnegative integers and $p_{1}+p_{2}\leq n$,
 denote by $\sigma_{l}(A)$ the $l$-order symbol of an operator $A$,
for equivariant Bismut Laplacian, an application of (2.1.4) in \cite{Wa1} shows that
\begin{align}
&\widetilde{Wres}\Big[\pi^{+} \big(\nabla^{S(TM)}_{V}
+\frac{3}{4}g(V,X)\big)\big(\nabla^{S(TM)}_{W}+\frac{3}{4}g(W,X)\big)H_{X} ^{-p_{1}}
\circ\pi^{+}H_{X}^{-p_{2}}\Big]\nonumber\\
=&\int_{M}\int_{|\xi|=1}\mathrm{Tr}_{S(TM)}
  \Big[\sigma_{-n} \Big(\big(\nabla^{S(TM)}_{V}
+\frac{3}{4}g(V,X)\big)\big(\nabla^{S(TM)}_{W}+\frac{3}{4}g(W,X)\big)\nonumber\\
  &\times
  H_{X}^{-p_{1}-p_{2}}\Big) \Big]\sigma(\xi)\texttt{d}x+\int_{\partial M}\Phi,
\end{align}
where
 \begin{align}
\Phi=&\int_{|\xi'|=1}\int_{-\infty}^{+\infty}\sum_{j,k=0}^{\infty}\sum \frac{(-i)^{|\alpha|+j+k+\ell}}{\alpha!(j+k+1)!}
\mathrm{Tr}_{S(TM)}\Big[\partial_{x_{n}}^{j}\partial_{\xi'}^{\alpha}\partial_{\xi_{n}}^{k}\sigma_{r}^{+}\Big
(\big(\nabla^{S(TM)}_{V}
+\frac{3}{4}g(V,X)\big)\nonumber\\
&\times\big(\nabla^{S(TM)}_{W}+\frac{3}{4}g(W,X)\big)H_{X}^{-p_{1}}\Big)(x',0,\xi',\xi_{n})
 \partial_{x'}^{\alpha}\partial_{\xi_{n}}^{j+1}\partial_{x_{n}}^{k}\sigma_{l}(H_{X})^{-p_{2}}(x',0,\xi',\xi_{n})\Big]
\texttt{d}\xi_{n}\sigma(\xi')\texttt{d}x' ,
\end{align}
and the sum is taken over $r-k+|\alpha|+\ell-j-1=-n,r\leq-p_{1},\ell\leq-p_{2}$.

\subsection{Spectral  Einstein functionals of the Bismut Laplacian for four dimensional manifolds with boundary}

In this section, we compute the noncommutative residue for  four dimension compact manifolds with boundary and get
spectral Einstein functionals of the Bismut Laplacian in this case.
Since $ \sigma_{-4}\Big(\big(\nabla^{S(TM)}_{V}+\frac{3}{4}g(V,X)\big)\times\big(\nabla^{S(TM)}_{W}+\frac{3}{4}g(W,X)\big)H_{X}^{-1}
  \circ H_{X}^{-1}\Big)\Big|_{M}$ has
the same expression as which in the case of manifolds without boundary,
so locally we can use Theorem 2.4 to compute the first term.
\begin{thm}
For  4-dimensional compact manifold $M$  and the equivariant Bismut Laplacian $H_{X}$, the spectral  Einstein functional equals to
\begin{align}
&Wres\Big(\big(\nabla^{S(TM)}_{V}+\frac{3}{4}g(V,X)\big)\big(\nabla^{S(TM)}_{W}+\frac{3}{4}g(W,X)\big)H_{X}^{-2}\Big)\nonumber\\
=&\frac{4\pi^{2}}{3} \int_{M}\big(Ric(V,W)-\frac{1}{2}sg(V,W)\big) vol_{g}
+3\pi^{2}\int_{M} \Big(g\big(\nabla_{V}^{TM}X,W\big)-g\big(\nabla_{W}^{TM}X,V\big) \Big)vol_{g}\nonumber\\
&+ \int_{M}\Big(\frac{1}{2}s^{g}+ \mathrm{div}^{g}(X)+ |X|^{2}
-2Tr\big( \mu^{S(TM)}(X)\big) \Big)g(V,W) vol_{g},
\end{align}
where $\mathrm{div}^{g}(X)$ denotes the divergence of $X$, and $s$ denotes the scalar curvature.
\end{thm}
Then we only need to compute $\int_{\partial M}\Phi$. The following Lemmas in terms of Bismut Laplacian  are very efficient for computing
the boundary term.
\begin{lem}\cite{Wa4} The symbols of the Dirac operator:
\begin{align}
&\sigma_{-2}(D^{-2})=|\xi|^{-2};\\
&\sigma_{-3}(D^{-2})=-\sqrt{-1}|\xi|^{-4}\xi_k(\Gamma^k-2\delta^k)-\sqrt{-1}|\xi|^{-6}2\xi^j\xi_\alpha\xi_\beta
\partial_jg^{\alpha\beta}.
\end{align}
\end{lem}

\begin{lem}\cite{Wa5}  The symbols of the Bismut Laplacian:
\begin{align}
&\sigma_{-2}(H_{X}^{-1})=|\xi|^{-2};\\
&\sigma_{-3}(H_{X}^{-2})=\sigma_{-3}(D^{-2})-\sum_{j=1}^{n}\sqrt{-1}|\xi|^{-4} \big (X_{j}-\frac{1}{2}\langle X ,  \partial_{j} \rangle   \big)\xi_j.
\end{align}
\end{lem}
\begin{lem}\label{lem3} The following identities hold:
\begin{align}
 &\sigma_{0}\Big(\big(\nabla^{S(TM)}_{V}+\frac{3}{4}g(V,X)\big)\big(\nabla^{S(TM)}_{W}+\frac{3}{4}g(W,X)\big)\Big)\nonumber\\
 =&V[A(W)]+A(V)A(W)
+ \frac{3}{4}g(V,X)A(W)+V\Big(\frac{3}{4}g(W,X)\Big) +\frac{3}{4}A(V)g(W,X) \nonumber\\
&
+ \frac{9}{16}g(V,X) g(W,X);\\
 &\sigma_{1}\Big(\big(\nabla^{S(TM)}_{V}+\frac{3}{4}g(V,X)\big)\big(\nabla^{S(TM)}_{W}+\frac{3}{4}g(W,X)\big)\Big)\nonumber\\
=&\sum_{j,l=1}^n V_j\frac{\partial W_l }{\partial x_j }\sqrt{-1}\xi_l
+\sqrt{-1}\sum_{j =1}^nA(W)V_j\xi_j+\sum_{l =1}^n\sqrt{-1}A(V)W_l\xi_l\nonumber\\
&+\sum_{l =1}^n\frac{3}{4}g(V,X)W_l\sqrt{-1}\xi_l+\sum_{j =1}^n \frac{3}{4}g(W,X)V_j\sqrt{-1} \xi_j;\\
 &\sigma_{2}\Big(\big(\nabla^{S(TM)}_{V}+\frac{3}{4}g(V,X)\big)\big(\nabla^{S(TM)}_{W}+\frac{3}{4}g(W,X)\big)\Big)
 =-\sum_{j,l=1}^nV_jW_l\xi_j\xi_l.
\end{align}
\end{lem}
\begin{proof}
In terms of the fixed orthonormal frame $\{e_1,\cdots,e_n\}$ and natural frames $\{\partial_1,\cdots,\partial_n\}$ of $TM$,
let $V=\sum_{j=1}^{n}V_{j}\partial_{x_{j}}$, $W=\sum_{l=1}^{n}W_{l}\partial_{x_{l}}$ and
\begin{align}
&\nabla^{S(TM)}_{W}=W+\frac{1}{4}\sum_{i,j}\langle \nabla^{L}_{W}e_{i}, e_{j}  \rangle c(e_{i})c(e_{j}) =W+A(W);\\
&\nabla^{S(TM)}_{V}=V+\frac{1}{4}\sum_{i,j}\langle \nabla^{L}_{V}e_{i}, e_{j}  \rangle c(e_{i})c(e_{j}) =V+A(V).
\end{align}
From  (2.18) we have
\begin{align}
\widetilde{\nabla}^{S(TM)}_{V}\widetilde{\nabla}^{S(TM)}_{W}
=&(\nabla^{S(TM)}_{V}+\frac{3}{4}g(V,X))(\nabla^{S(TM)}_{W}+\frac{3}{4}g(W,X))\nonumber\\
=&(V+A(V)+\frac{3}{4}g(V,X))(W+A(W)+\frac{3}{4}g(W,X))\nonumber\\
=&V(W)+V(A(W))+A(W)V+A(V)W+A(V)A(W)+\frac{3}{4}g(V,X)W\nonumber\\
&+\frac{3}{4}g(V,X)A(W)+V(\frac{3}{4}g(W,X))+\frac{3}{4}g(W,X)V\nonumber\\
&+A(V)\frac{3}{4}g(W,X)+\frac{9}{16}g(V,X)g(W,X).
\end{align}
We note that
\begin{align}
VW=\sum_{j=1}^{n}V_{j}\partial_{x_{j}}\Big(\sum_{l=1}^{n}W_{l}\partial_{x_{l}}\Big)
=\sum_{j,l=1}^{n}\Big(V_{j}W_{l}\frac{\partial}{\partial x_{j}}\frac{\partial}{\partial x_{l}}
+V_{j}\frac{\partial W_{l}}{\partial x_{j}}\frac{\partial}{\partial x_{l}}   \Big).
\end{align}
In view of (3.20) and (3.21), and the proof of the lemma is complete.
\end{proof}

Since $\Phi$ is a global form on $\partial M$, so for any fixed point $x_{0}\in\partial M$, we can choose the normal coordinates
$U$ of $x_{0}$ in $\partial M$(not in $M$) and compute $\Phi(x_{0})$ in the coordinates $\widetilde{U}=U\times [0,1)$ and the metric
$\frac{1}{h(x_{n})}g^{\partial M}+\texttt{d}x _{n}^{2}$. The dual metric of $g^{\partial M}$ on $\widetilde{U}$ is
$\frac{1}{\tilde{h}(x_{n})}g^{\partial M}+\texttt{d}x _{n}^{2}.$ Write
$g_{ij}^{M}=g^{M}(\frac{\partial}{\partial x_{i}},\frac{\partial}{\partial x_{j}})$;
$g^{ij}_{M}=g^{M}(d x_{i},dx_{j})$, then

\begin{equation}
[g_{i,j}^{M}]=
\begin{bmatrix}\frac{1}{h( x_{n})}[g_{i,j}^{\partial M}]&0\\0&1\end{bmatrix};\quad
[g^{i,j}_{M}]=\begin{bmatrix} h( x_{n})[g^{i,j}_{\partial M}]&0\\0&1\end{bmatrix},
\end{equation}
and
\begin{equation}
\partial_{x_{s}} g_{ij}^{\partial M}(x_{0})=0,\quad 1\leq i,j\leq n-1;\quad g_{i,j}^{M}(x_{0})=\delta_{ij}.
\end{equation}

Let $\{e_{1},\cdots, e_{n-1}\}$ be an orthonormal frame field in $U$ about $g^{\partial M}$ which is parallel along geodesics and
$e_{i}=\frac{\partial}{\partial x_{i}}(x_{0})$, then $\{\widetilde{e_{1}}=\sqrt{h(x_{n})}e_{1}, \cdots,
\widetilde{e_{n-1}}=\sqrt{h(x_{n})}e_{n-1},\widetilde{e_{n}}=dx_{n}\}$ is the orthonormal frame field in $\widetilde{U}$ about $g^{M}.$
Locally $S(TM)|\widetilde{U}\cong \widetilde{U}\times\wedge^{*}_{C}(\frac{n}{2}).$ Let $\{f_{1},\cdots,f_{n}\}$ be the orthonormal basis of
$\wedge^{*}_{C}(\frac{n}{2})$. Take a spin frame field $\sigma: \widetilde{U}\rightarrow Spin(M)$ such that
$\pi\sigma=\{\widetilde{e_{1}},\cdots, \widetilde{e_{n}}\}$ where $\pi: Spin(M)\rightarrow O(M)$ is a double covering, then
$\{[\sigma, f_{i}], 1\leq i\leq 4\}$ is an orthonormal frame of $S(TM)|_{\widetilde{U}}.$ In the following, since the global form $\Phi$
is independent of the choice of the local frame, so we can compute $\texttt{tr}_{S(TM)}$ in the frame $\{[\sigma, f_{i}], 1\leq i\leq 4\}$.
Let $\{E_{1},\cdots,E_{n}\}$ be the canonical basis of $R^{n}$ and
$c(E_{i})\in cl_{C}(n)\cong Hom(\wedge^{*}_{C}(\frac{n}{2}),\wedge^{*}_{C}(\frac{n}{2}))$ be the Clifford action. By \cite{Wa3}, then

\begin{equation}
c(\widetilde{e_{i}})=[(\sigma,c(E_{i}))]; \quad c(\widetilde{e_{i}})[(\sigma, f_{i})]=[\sigma,(c(E_{i}))f_{i}]; \quad
\frac{\partial}{\partial x_{i}}=[(\sigma,\frac{\partial}{\partial x_{i}})],
\end{equation}
then we have $\frac{\partial}{\partial x_{i}}c(\widetilde{e_{i}})=0$ in the above frame. By Lemma 2.2 in \cite{Wa3}, we have
\begin{lem}\cite{Wa3}
With the metric $g^{M}$ on $M$ near the boundary
\begin{eqnarray}
\partial_{x_j}(|\xi|_{g^M}^2)(x_0)&=&\left\{
       \begin{array}{c}
        0,  ~~~~~~~~~~ ~~~~~~~~~~ ~~~~~~~~~~~~~{\rm if }~j<n; \\[2pt]
       h'(0)|\xi'|^{2}_{g^{\partial M}},~~~~~~~~~~~~~~~~~~~~~{\rm if }~j=n.
       \end{array}
    \right. \\
\partial_{x_j}[c(\xi)](x_0)&=&\left\{
       \begin{array}{c}
      0,  ~~~~~~~~~~ ~~~~~~~~~~ ~~~~~~~~~~~~~{\rm if }~j<n;\\[2pt]
\frac{h'(0)}{2} c(\xi')(x_{0}), ~~~~~~~~~~~~~~~~~~~{\rm if }~j=n.
       \end{array}
    \right.
\end{eqnarray}
\end{lem}

Now we  need to compute $\int_{\partial M} \Phi$. When $n=4$, then ${\rm tr}_{S(TM)}[{\rm \texttt{id}}]={\rm dim}(\wedge^*(\mathbb{R}^2))=4$, the sum is taken over $
r+l-k-j-|\alpha|=-3,~~r\leq 0,~~l\leq-2,$ then we have the following five cases:

 {\bf case a)~I)}~$r=0,~l=-2,~k=j=0,~|\alpha|=1$.

 By (3.9), we get
\begin{align}
\label{b24}
\Phi_1=&-\int_{|\xi'|=1}\int^{+\infty}_{-\infty}\sum_{|\alpha|=1}
\mathrm{Tr}[\partial^\alpha_{\xi'}\pi^+_{\xi_n}\sigma_{0}
\Big(\big(\nabla^{S(TM)}_{V}+\frac{3}{4}g(V,X)\big)\big(\nabla^{S(TM)}_{W}+\frac{3}{4}g(W,X)\big)H_{X}^{-1}\Big) \nonumber\\
&\times
 \partial^\alpha_{x'}\partial_{\xi_n}\sigma_{-2}(H_{X}^{-1})](x_0)\mathrm{d}\xi_n\sigma(\xi')\mathrm{d}x'.
\end{align}

By Lemma 3.6, for $i<n$, then
\begin{equation}
\label{b25}
\partial_{x_i}\sigma_{-2}(H_{X}^{-1})(x_0)=
\partial_{x_i}(|\xi|^{-2})(x_0)=
-\frac{\partial_{x_i}(|\xi|^{2})(x_0)}{|\xi|^4}=0,
\end{equation}
 so $\Phi_1=0$.

 {\bf case a)~II)}~$r=0,~l=-2,~k=|\alpha|=0,~j=1$.

 By (3.9), we get
\begin{align}
\label{b26}
\Phi_2=&-\frac{1}{2}\int_{|\xi'|=1}\int^{+\infty}_{-\infty}
\mathrm{Tr}[\partial_{x_n}\pi^+_{\xi_n}
\sigma_{0}\Big(\big(\nabla^{S(TM)}_{V}+\frac{3}{4}g(V,X)\big)\big(\nabla^{S(TM)}_{W}+\frac{3}{4}g(W,X)\big)H_{X}^{-1}\Big) \nonumber\\
&\times
\partial_{\xi_n}^2\sigma_{-2}(H_{X}^{-1})](x_0)\mathrm{d}\xi_n\sigma(\xi')\mathrm{d}x'.
\end{align}
By Lemma 3.6, we have
\begin{eqnarray}\label{b237}
\partial_{\xi_n}^2\sigma_{-2}(H_{X}^{-1})(x_0)=\partial_{\xi_n}^2(|\xi|^{-2})(x_0)=\frac{6\xi_n^2-2}{(1+\xi_n^2)^3}.
\end{eqnarray}
It follows that
\begin{align}
&\partial_{x_n}\sigma_{0}\Big(\big(\nabla^{S(TM)}_{V}+\frac{3}{4}g(V,X)\big)\big(\nabla^{S(TM)}_{W}
+\frac{3}{4}g(W,X)\big)H_{X}^{-1}\Big)(x_0)\nonumber\\
=&\partial_{x_n}\Big[\sigma_{2}\Big(\big(\nabla^{S(TM)}_{V}+\frac{3}{4}g(V,X)\big)\big(\nabla^{S(TM)}_{W}
+\frac{3}{4}g(W,X)\big)\Big)\sigma_{-2}(H_{X}^{-1})\Big](x_0)\nonumber\\
=&\partial_{x_n}\Big(-\sum_{j,l=1}^nV_jW
_l\xi_j\xi_l|\xi|^{-2}\Big)(x_0)\nonumber\\
=&\frac{-1}{(1+\xi_n^2) }\sum_{j,l=1}^n \partial_{x_n}\big(V_jW_l \big)\xi_j\xi_l
+\frac{1}{(1+\xi_n^2)^2}\sum_{j,l=1}^n V_jW_l\xi_j\xi_lh'(0)\nonumber\\
=&\frac{-1}{(1+\xi_n^2) }\sum_{j =1}^{n-1} \partial_{x_n}\big(V_jW_n \big)\xi_j\xi_n
+\frac{1}{(1+\xi_n^2) }\sum_{ l=1}^{n-1} \partial_{x_n}\big(V_nW_l \big)\xi_j\xi_l  \nonumber\\
&+\frac{-1}{(1+\xi_n^2) }\sum_{j,l=1}^{n-1} \partial_{x_n}\big(V_jW_l \big)\xi_j\xi_l
+\frac{1}{(1+\xi_n^2) } \partial_{x_n}\big(V_nW_n \big)\xi_n\xi_n  \nonumber\\
&+\frac{-1}{(1+\xi_n^2)^2 }\sum_{j =1}^{n-1}  V_jW_n\xi_j\xi_lh'(0)
+\frac{1}{(1+\xi_n^2)^2}\sum_{l=1}^{n-1}  V_nW_l\xi_j\xi_lh'(0) \nonumber\\
&+\frac{-1}{(1+\xi_n^2)^2 }\sum_{j,l=1}^{n-1}  V_jW_l\xi_j\xi_lh'(0)
+\frac{1}{(1+\xi_n^2)^2} V_nW_n \xi_n\xi_nh'(0).
\end{align}
We note that $i<n,~\int_{|\xi'|=1}\xi_{i_{1}}\xi_{i_{2}}\cdots\xi_{i_{2d+1}}\sigma(\xi')=0$,
so we omit some items that have no contribution for computing {\bf case a)~II)}.
By the Cauchy integral formula, we obtain
\begin{align}\label{b28}
&\pi^+_{\xi_n}\Big(\frac{-1}{(1+\xi_n^2) }\sum_{j,l=1}^{n-1} \partial_{x_n}\big(V_jW_l \big)\xi_j\xi_l
+\frac{1}{(1+\xi_n^2) } \partial_{x_n}\big(V_nW_n \big)\xi_n\xi_n  \nonumber\\
 &+\frac{-1}{(1+\xi_n^2)^2 }\sum_{j,l=1}^{n-1}  V_jW_l\xi_j\xi_lh'(0)
+\frac{1}{(1+\xi_n^2)^2} V_nW_n \xi_n\xi_nh'(0)\Big)(x_0)\nonumber\\
 =& \frac{i }{2(\xi_n-i) }\sum_{j,l=1}^{n-1} \partial_{x_n}\big(V_jW_l \big)\xi_j\xi_l
- \frac{i }{2(\xi_n-i) } \partial_{x_n}\big(V_nW_n \big)  \nonumber\\
 &- \frac{2+i\xi_n}{4(\xi_n-i)^2}\sum_{j,l=1}^{n-1}  V_jW_l\xi_j\xi_lh'(0)
-\frac{1}{4(\xi_n-i)^2} V_nW_n  h'(0).
\end{align}
From (3.30) and (3.32), we obtain
\begin{align}
\label{b26}
\Phi_2=&-\frac{1}{2}\int_{|\xi'|=1}\int^{+\infty}_{-\infty}
\mathrm{Tr}[\partial_{x_n}\pi^+_{\xi_n}
\sigma_{0}\Big(\big(\nabla^{S(TM)}_{V}+\frac{3}{4}g(V,X)\big)\big(\nabla^{S(TM)}_{W}+\frac{3}{4}g(W,X)\big)H_{X}^{-1}\Big) \nonumber\\
&\times
\partial_{\xi_n}^2\sigma_{-2}(H_{X}^{-1})](x_0)\mathrm{d}\xi_n\sigma(\xi')\mathrm{d}x'\nonumber\\
=&-\frac{1}{2}\int_{|\xi'|=1}\int^{+\infty}_{-\infty}
\frac{i(3\xi_n^{2}-1) }{ (\xi_n-i)^{4} (\xi_n+i)^{3}}\sum_{j,l=1}^{n-1} \partial_{x_n}\big(V_jW_l \big)\xi_j\xi_l
                                                           \mathrm{Tr}^{S(TM)}(\mathrm{id})\mathrm{d}\xi_n\sigma(\xi')\mathrm{d}x'\nonumber\\
&-\frac{1}{2}\int_{|\xi'|=1}\int^{+\infty}_{-\infty} \frac{-(2+i\xi_n)(3\xi_n^{2}-1) }{ 2(\xi_n-i)^{5} (\xi_n+i)^{3}}
 \sum_{j,l=1}^{n-1}  V_jW_l\xi_j\xi_lh'(0)  \mathrm{Tr}^{S(TM)}(\mathrm{id})\mathrm{d}\xi_n\sigma(\xi')\mathrm{d}x'\nonumber\\
 &-\frac{1}{2}\int_{|\xi'|=1}\int^{+\infty}_{-\infty}\frac{-i(3\xi_n^{2}-1) }{ (\xi_n-i)^{4} (\xi_n+i)^{3}}\partial_{x_n}\big(V_nW_n \big)
                                                    \mathrm{Tr}^{S(TM)}(\mathrm{id})  \mathrm{d}\xi_n\sigma(\xi')\mathrm{d}x'\nonumber\\
  &-\frac{1}{2}\int_{|\xi'|=1}\int^{+\infty}_{-\infty}\frac{-(3\xi_n^{2}-1) }{ 2(\xi_n-i)^{4} (\xi_n+i)^{3}} V_nW_n  h'(0)
                                                  \mathrm{Tr}^{S(TM)} (\mathrm{id})\mathrm{d}\xi_n\sigma(\xi')\mathrm{d}x'\nonumber\\
 =& \Big(-\frac{\pi^{2}}{3}\sum_{j =1}^{n-1} \partial_{x_n}\big(V_jW_j \big)
 -\frac{5}{6}\pi^{2} \sum_{j =1}^{n-1}  V_jW_jh'(0) -2\pi^{2}\partial_{x_n}\big(V_nW_n \big)
 +i\pi^{2} V_nW_n  h'(0) \Big)  \mathrm{d}x'.
\end{align}

  {\bf case a)~III)}~$r=0,~l=-2,~j=|\alpha|=0,~k=1$.

By (3.9), we get
\begin{align}\label{36}
\Phi_3=&-\frac{1}{2}\int_{|\xi'|=1}\int^{+\infty}_{-\infty}
\mathrm{Tr}\Big[\partial_{\xi_n}\pi^+_{\xi_n}
\sigma_{0}\Big(\big(\nabla^{S(TM)}_{V}+\frac{3}{4}g(V,X)\big)\big(\nabla^{S(TM)}_{W}+\frac{3}{4}g(W,X)\big)H_{X}^{-1}\Big) \nonumber\\
&\times\partial_{\xi_n} \partial_{x_n}\sigma_{-2}(H_{X}^{-1})\Big](x_0)\mathrm{d}\xi_n\sigma(\xi')\mathrm{d}x'.
\end{align}
 By Lemma 3.6, we have
\begin{eqnarray}\label{37}
\partial_{x_n}\sigma_{-2}(H_{X}^{-1})(x_0)|_{|\xi'|=1}
=-\frac{h'(0)}{(1+\xi_n^2)^2}.
\end{eqnarray}
An easy calculation gives
\begin{eqnarray}\label{37}
\partial_{\xi_n}\partial_{x_n}\sigma_{-2}(H_{X}^{-1})(x_0)|_{|\xi'|=1}
=-\frac{4\xi_nh'(0)}{(1+\xi_n^2)^3}.
\end{eqnarray}
By the Cauchy integral formula, we obtain
\begin{align}\label{38}
&\pi^+_{\xi_n}\sigma_{0}\Big(\big(\nabla^{S(TM)}_{V}+\frac{3}{4}g(V,X)\big)\big(\nabla^{S(TM)}_{W}
+\frac{3}{4}g(W,X)\big)H_{X}^{-1}\Big)(x_0)|_{|\xi'|=1}\nonumber\\
=&\frac{i}{2(\xi_n-i)}\sum_{j,l=1}^{n-1}V_jW_l\xi_j\xi_l+\frac{1}{2(\xi_n-i)}V_nW_n-\frac{1}{2(\xi_n-i)}\sum_{j=1}^{n-1}V_jW_n\xi_j\nonumber\\
&-\frac{1}{2(\xi_n-i)}\sum_{l=1}^{n-1}V_nW_l\xi_l.
\end{align}
Also, straightforward computations yield
\begin{align}\label{38}
&\partial_{\xi_n}\pi^+_{\xi_n}\sigma_{0}\Big(\big(\nabla^{S(TM)}_{V}+\frac{3}{4}g(V,X)\big)\big(\nabla^{S(TM)}_{W}
+\frac{3}{4}g(W,X)\big)H_{X}^{-1}\Big)(x_0)|_{|\xi'|=1}\nonumber\\
=&\frac{-i}{2(\xi_n-i)^{2}}\sum_{j,l=1}^{n-1}V_jW_l\xi_j\xi_l-\frac{1}{2(\xi_n-i)^{2}}V_nW_n
+\frac{1}{2(\xi_n-i)^{2}}\sum_{j=1}^{n-1}V_jW_n\xi_j\nonumber\\
&+\frac{1}{2(\xi_n-i)^{2}}\sum_{l=1}^{n-1}V_nW_l\xi_l.
\end{align}
Therefore, we get
\begin{align}\label{36}
\Phi_3=&-\frac{1}{2}\int_{|\xi'|=1}\int^{+\infty}_{-\infty}
\mathrm{Tr}\Big[\partial_{\xi_n}\pi^+_{\xi_n}
\sigma_{0}\Big(\big(\nabla^{S(TM)}_{V}+\frac{3}{4}g(V,X)\big)\big(\nabla^{S(TM)}_{W}+\frac{3}{4}g(W,X)\big)H_{X}^{-1}\Big) \nonumber\\
&\times\partial_{\xi_n} \partial_{x_n}\sigma_{-2}(H_{X}^{-1})\Big](x_0)\mathrm{d}\xi_n\sigma(\xi')\mathrm{d}x' \nonumber\\
&=-2\sum_{j,l=1}^{n-1}V_jW_lh'(0)\frac{4\pi}{3} \int_{\Gamma^{+}}\frac{i}{(\xi_n-i)^5(\xi_n+i)^2} d\xi_{n}dx'
+2V_nW_nh'(0)\Omega_3\int_{\Gamma^{+}}\frac{1}{(\xi_n-i)^5(\xi_n+i)^2}d\xi_{n}dx'\nonumber\\
&=-2\sum_{j,l=1}^{n-1}V_jW_lh'(0)\frac{4\pi}{3} \frac{2\pi i}{4!}\left[\frac{i}{(\xi_n+i)^2}\right]^{(4)}\bigg|_{\xi_n=i}dx'
+2V_nW_nh'(0)\Omega_3\frac{2\pi i}{4!}\left[\frac{1}{(\xi_n+i)^2}\right]^{(4)}\bigg|_{\xi_n=i}dx'\nonumber\\
&=\left(\frac{5\pi^{2}}{12}\sum_{j=1}^{n-1}V_jW_j+\frac{5i\pi^{2}}{4}V_nW_n\right)h'(0) dx'.
\end{align}

 {\bf case b)}~$r=0,~l=-3,~k=j=|\alpha|=0$.

 By (3.9), we get
 \begin{align}\label{36}
\Phi_4=&-i\int_{|\xi'|=1}\int^{+\infty}_{-\infty}
\mathrm{Tr}\Big[ \pi^+_{\xi_n}
\sigma_{0}\Big(\big(\nabla^{S(TM)}_{V}+\frac{3}{4}g(V,X)\big)\big(\nabla^{S(TM)}_{W}+\frac{3}{4}g(W,X)\big)H_{X}^{-1}\Big) \nonumber\\
&\times\partial_{\xi_n} \sigma_{-3}(H_{X}^{-1})\Big](x_0)\mathrm{d}\xi_n\sigma(\xi')\mathrm{d}x' \nonumber\\
=&i\int_{|\xi'|=1}\int^{+\infty}_{-\infty}
\mathrm{Tr}\Big[ \partial_{\xi_n} \pi^+_{\xi_n}
\sigma_{0}\Big(\big(\nabla^{S(TM)}_{V}+\frac{3}{4}g(V,X)\big)\big(\nabla^{S(TM)}_{W}+\frac{3}{4}g(W,X)\big)H_{X}^{-1}\Big) \nonumber\\
&\times\sigma_{-3}(H_{X}^{-1})\Big](x_0)\mathrm{d}\xi_n\sigma(\xi')\mathrm{d}x'.
\end{align}
By Lemma 3.4 and Lemma 3.5, we get
\begin{align}\label{45}
&\partial_{\xi_n}\pi^+_{\xi_n}\sigma_{0}\Big(\big(\nabla^{S(TM)}_{V}+\frac{3}{4}g(V,X)\big)
\big(\nabla^{S(TM)}_{W}+\frac{3}{4}g(W,X)\big)H_{X}^{-1}\Big)(x_0)|_{|\xi'|=1}\nonumber\\
=&-\frac{i}{2(\xi_n-i)^2}\sum_{j,l=1}^{n-1}V_jW_l\xi_j\xi_l-\frac{1}{2(\xi_n-i)^2}V_nW_n\nonumber\\
&+\frac{1}{2(\xi_n-i)^2}\sum_{j=1}^{n-1}V_jW_n\xi_j+\frac{1}{2(\xi_n-i)^2}\sum_{l=1}^{n-1}V_nW_l\xi_l.
\end{align}
By Lemma 3.3 and Lemma 3.4, we have
\begin{align}\label{43}
\sigma_{-3}(H_{X}^{-1})(x_0)|_{|\xi'|=1}
=&-\frac{i}{(1+\xi_n^2)^2}\left(-\frac{1}{2}h'(0)\sum_{k<n}\xi_nc(e_k)c(e_n)+\frac{5}{2}h'(0)\xi_n\right)
-\frac{2ih'(0)\xi_n}{(1+\xi_n^2)^3}\nonumber\\
&-\sqrt{-1}|\xi|^{-4} \big (X_{j}-\frac{1}{2}\langle X ,  \partial_{j} \rangle   \big)\xi_j\nonumber\\
=:&\sigma_{-3}(D^{-2})(x_0)|_{|\xi'|=1}-\sum_{j=1}^{n}\sqrt{-1}|\xi|^{-4} \big (X_{j}-\frac{1}{2}\langle X ,  \partial_{j} \rangle   \big)\xi_j.
\end{align}
We note that $i<n,~\int_{|\xi'|=1}\xi_{i_{1}}\xi_{i_{2}}\cdots\xi_{i_{2d+1}}\sigma(\xi')=0$,
and $\sum_{i\neq s\neq t}Tr[A_{ist}c(\widetilde{e_{i}})c(\widetilde{e_{s}})c(\widetilde{e_{t}})c(\xi')]=0$,
so we omit some items that have no contribution for computing {\bf case b)}. Then
\begin{align}\label{41}
&i\int_{|\xi'|=1}\int^{+\infty}_{-\infty}
\mathrm{Tr}\Big[ \partial_{\xi_n} \pi^+_{\xi_n}
\sigma_{0}\Big(\big(\nabla^{S(TM)}_{V}+\frac{3}{4}g(V,X)\big)\big(\nabla^{S(TM)}_{W}+\frac{3}{4}g(W,X)\big)H_{X}^{-1}\Big) \nonumber\\
&\times\sigma_{-3}(D^{-2})\Big](x_0)\mathrm{d}\xi_n\sigma(\xi')\mathrm{d}x'\nonumber\\
=&-i\sum_{j,l=1}^{n-1}V_jW_lh'(0)\frac{4\pi}{3}\int_{\Gamma^{+}}
\frac{5\xi_n^2-5+4\xi_n}{(\xi_n-i)^5(\xi_n+i)^2} d\xi_{n}dx'+iV_nW_nh'(0)\Omega_3
\int_{\Gamma^{+}}\frac{5\xi_n^3-\xi_n}{(\xi_n-i)^5(\xi_n+i)^2}d\xi_{n}dx'\nonumber\\
=&-i\sum_{j,l=1}^{n-1}V_jW_lh'(0)\frac{4\pi}{3}\frac{2\pi i}{4!}\left[\frac{5\xi_n^2-5+4\xi_n}{(\xi_n+i)^2}\right]^{(4)}\bigg|_{\xi_n=i}dx'
+iV_nW_nh'(0)\Omega_3\frac{2\pi i}{4!}\left[\frac{5\xi_n^3-\xi_n}{(\xi_n+i)^2}\right]^{(4)}\bigg|_{\xi_n=i}dx'\nonumber\\
=&\Big(\frac{(1-5i)\pi^{2}}{12}\sum_{j=1}^{n-1}V_jW_j+\frac{11i}{4}\pi^{2}V_nW_n\Big)h'(0) dx'.
\end{align}
On the other hand,
\begin{align}\label{41}
&i\int_{|\xi'|=1}\int^{+\infty}_{-\infty}
\mathrm{Tr}\Big[ \partial_{\xi_n} \pi^+_{\xi_n}
\sigma_{0}\Big(\big(\nabla^{S(TM)}_{V}+\frac{3}{4}g(V,X)\big)\big(\nabla^{S(TM)}_{W}+\frac{3}{4}g(W,X)\big)H_{X}^{-1}\Big) \nonumber\\
&\times\Big(-\sum_{j=1}^{n}\sqrt{-1}|\xi|^{-4} \big (X_{j}-\frac{1}{2}\langle X ,  \partial_{j} \rangle   \big)\xi_j\Big)\Big]
(x_0)\mathrm{d}\xi_n\sigma(\xi')\mathrm{d}x'\nonumber\\
=&i\int_{|\xi'|=1}\int^{+\infty}_{-\infty}
\frac{-i }{ 2(\xi_n-i)^{4} (\xi_n+i)^{2}}\sum_{j,l=1}^{n}\xi_{j}\xi_{l}
\Big(-\sum_{j=1}^{n-1}\sqrt{-1}|\xi|^{-4} \big (X_{j}-\frac{1}{2}\langle X ,  \partial_{j} \rangle   \big)\xi_j\Big)\nonumber\\
&\times V_{j}W_{n}  \mathrm{Tr}^{S(TM)}(\mathrm{id})\mathrm{d}\xi_n\sigma(\xi')\mathrm{d}x'\nonumber\\
&+i\int_{|\xi'|=1}\int^{+\infty}_{-\infty} \frac{-i }{ 2(\xi_n-i)^{4} (\xi_n+i)^{2}}\sum_{j,l=1}^{n}\xi_{j}\xi_{l}
\Big(-\sum_{l=1}^{n-1}\sqrt{-1}|\xi|^{-4} \big (X_{j}-\frac{1}{2}\langle X ,  \partial_{j} \rangle   \big)\xi_j\Big)\nonumber\\
&\times V_{n}W_{l}
 \mathrm{Tr}^{S(TM)}(\mathrm{id})\mathrm{d}\xi_n\sigma(\xi')\mathrm{d}x'\nonumber\\
 &+i\int_{|\xi'|=1}\int^{+\infty}_{-\infty}\frac{-\xi_n }{ 2(\xi_n-i)^{4} (\xi_n+i)^{2}}\sum_{j,l=1}^{n}\xi_{j}\xi_{l}
\Big(- \sqrt{-1}|\xi|^{-4} \big (X_{n}-\frac{1}{2}\langle X ,  \partial_{n} \rangle   \big) \Big)\nonumber\\
&\times \sum_{j =1}^{n-1}V_{j}W_{l} \mathrm{Tr}^{S(TM)}(\mathrm{id})  \mathrm{d}\xi_n\sigma(\xi')\mathrm{d}x'\nonumber\\
  &+i\int_{|\xi'|=1}\int^{+\infty}_{-\infty}\frac{i\xi_n  }{ 2(\xi_n-i)^{4} (\xi_n+i)^{2}}
 \Big(- \sqrt{-1}|\xi|^{-4} \big (X_{n}-\frac{1}{2}\langle X ,  \partial_{n} \rangle   \big) \Big)\nonumber\\
 &\times V_{n}W_{n} \mathrm{Tr}^{S(TM)} (\mathrm{id})\mathrm{d}\xi_n\sigma(\xi')\mathrm{d}x'\nonumber\\
 =& \Big(-\frac{2\pi^{2}}{3}\sum_{j =1}^{n-1} \big (X_{j}-\frac{1}{2}\langle X ,  \partial_{j} \rangle   \big)V_{j}W_{n}
 -\frac{2\pi^{2}}{3}\sum_{l =1}^{n-1} \big (X_{l}-\frac{1}{2}\langle X ,  \partial_{l} \rangle   \big)V_{n}W_{l}\nonumber\\
 &+\frac{ \pi^{2}}{3}\sum_{j  =1}^{n-1}V_{j}W_{j} \big (X_{n}-\frac{1}{2}\langle X ,  \partial_{n} \rangle   \big)
  -i\pi^{2} V_{n}W_{n} \big (X_{n}-\frac{1}{2}\langle X ,  \partial_{n} \rangle   \big)\Big)  \mathrm{d}x'.
\end{align}
Summing up (3.42),(3.43) leads to the result of $\Phi_4$.

 {\bf  case c)}~$r=-1,~\ell=-2,~k=j=|\alpha|=0$.

By (3.9), we get
\begin{align}\label{36}
\Phi_5=&-i\int_{|\xi'|=1}\int^{+\infty}_{-\infty}
\mathrm{Tr}\Big[ \pi^+_{\xi_n}
\sigma_{-1}\Big(\big(\nabla^{S(TM)}_{V}+\frac{3}{4}g(V,X)\big)\big(\nabla^{S(TM)}_{W}+\frac{3}{4}g(W,X)\big)H_{X}^{-1}\Big) \nonumber\\
&\times\partial_{\xi_n} \sigma_{-2}(H_{X}^{-1})\Big](x_0)\mathrm{d}\xi_n\sigma(\xi')\mathrm{d}x'.
\end{align}
By Lemma 3.4, we have
\begin{align}\label{62}
\partial_{\xi_n}\sigma_{-2}(H_{X}^{-1})(x_0)|_{|\xi'|=1}=-\frac{2\xi_n}{(\xi_n^2+1)^2}.
\end{align}
On the other hand,
write
 \begin{eqnarray}
D_x^{\alpha}&=(-i)^{|\alpha|}\partial_x^{\alpha};
~\sigma(D_t)=p_1+p_0;
~ \sigma(D_t)^{-1} =\sum_{j=1}^{\infty}q_{-j}.
\end{eqnarray}
then we have
\begin{align}
1=\sigma(D\circ D^{-1})&=\sum_{\alpha}\frac{1}{\alpha!}\partial^{\alpha}_{\xi}[\sigma(D)]
D_x^{\alpha}[\sigma(D^{-1})]\nonumber\\
&=(p_1+p_0)(q_{-1}+q_{-2}+q_{-3}+\cdots)\nonumber\\
&~~~+\sum_j(\partial_{\xi_j}p_1+\partial_{\xi_j}p_0)(
D_{x_j}q_{-1}+D_{x_j}q_{-2}+D_{x_j}q_{-3}+\cdots)\nonumber\\
&=p_1q_{-1}+(p_1q_{-2}+p_0q_{-1}+\sum_j\partial_{\xi_j}p_1D_{x_j}q_{-1})+\cdots,
\end{align}
then
\begin{equation}
q_{-1}=p_1^{-1};~q_{-2}=-p_1^{-1}[p_0p_1^{-1}+\sum_j\partial_{\xi_j}p_1D_{x_j}(p_1^{-1})].
\end{equation}
By the composition formula of pseudodifferential operators, we obtain
\begin{align}
&\sigma_{-1}\Big(\big(\nabla^{S(TM)}_{V}+\frac{3}{4}g(V,X)\big)\big(\nabla^{S(TM)}_{W}
+\frac{3}{4}g(W,X)\big)H_{X}^{-1}\Big)(x_0)|_{|\xi'|=1}\nonumber\\
=&
\sigma_{2}\Big(\big(\nabla^{S(TM)}_{V}+\frac{3}{4}g(V,X)\big)
\big(\nabla^{S(TM)}_{W}+\frac{3}{4}g(W,X)\big) \Big)\sigma_{-3}(H_{X}^{-1})\nonumber\\
&+\sigma_{1}\Big(\big(\nabla^{S(TM)}_{V}+\frac{3}{4}g(V,X)\big)
\big(\nabla^{S(TM)}_{W}+\frac{3}{4}g(W,X)\big) \Big)\sigma_{-2}(H_{X}^{-1})\nonumber\\
&+\sum_{j=1}^{n}\partial_{\xi_{j}}\Big[\sigma_{2}
\Big(\big(\nabla^{S(TM)}_{V}+\frac{3}{4}g(V,X)\big)\big(\nabla^{S(TM)}_{W}+\frac{3}{4}g(W,X)\big)\Big)\Big]
D_{x_{j}}\big[\sigma_{-2}(H_{X}^{-1})\big]\nonumber\\
=:& B_{1}+B_{2}+B_{3}.
\end{align}

(i) Explicit representation  of $B_{1}$,
\begin{align}
B_{1}=&\sigma_{2}\Big(\big(\nabla^{S(TM)}_{V}+\frac{3}{4}g(V,X)\big)
\big(\nabla^{S(TM)}_{W}+\frac{3}{4}g(W,X)\big) \Big)\sigma_{-3}(H_{X}^{-1})(x_0)|_{|\xi'|=1}\nonumber\\
=&-\sum_{j,l=1}^{n}V_jW_l\xi_j\xi_l\times\Big(-\sqrt{-1}|\xi|^{-4}\xi_k(\Gamma^k-2\delta^k)
-\sqrt{-1}|\xi|^{-6}2\xi^j\xi_\alpha\xi_\beta\partial_jg^{\alpha\beta}\nonumber\\
&-\sum_{j=1}^{n}\sqrt{-1}|\xi|^{-4} \big (X_{j}-\frac{1}{2}\langle X ,  \partial_{j} \rangle   \big)\xi_j\Big)\nonumber\\
=&-\sum_{j,l=1}^{n}V_jW_l\xi_j\xi_l\times\Big( -\frac{i}{(1+\xi_n^2)^2}
\big(-\frac{1}{2}h'(0)\sum_{k<n}\xi_nc(e_k)c(e_n)+\frac{5}{2}h'(0)\xi_n\big)
-\frac{2ih'(0)\xi_n}{(1+\xi_n^2)^3}\nonumber\\
&-\sum_{j=1}^{n}\sqrt{-1}|\xi|^{-4} \big (X_{j}-\frac{1}{2}\langle X ,  \partial_{j} \rangle   \big)\xi_j\Big)\nonumber\\
=&-\sum_{j,l=1}^{n}V_jW_l\xi_j\xi_l\times  \frac{ih'(0)}{2(1+\xi_n^2)^2}
 \sum_{k<n}\xi_nc(e_k)c(e_n)
 -\sum_{j,l=1}^{n}V_jW_l\xi_j\xi_l\times  \frac{ -9i\xi_n-5i\xi_n^{3 } }{2(1+\xi_n^{2})^{3}}h'(0)   \nonumber\\
&-\sum_{j,l=1}^{n}V_jW_l\xi_j\xi_l\times\Big( -\sum_{j=1}^{n}\sqrt{-1}|\xi|^{-4} \big (X_{j}-\frac{1}{2}
\langle X ,  \partial_{j} \rangle   \big)\xi_j\Big)\nonumber\\
=&B_{1}(1)+B_{1}(2)+B_{1}(3).
\end{align}
We note that $\mathrm{Tr}(  \sum_{k<n} c(e_k)c(e_n)  )=0$, then
 we omit  $B_{1}(1)$ which has no contribution for computing {\bf case c)}. By the Cauchy integral formula, we obtain
\begin{align}\label{36}
 &-i\int_{|\xi'|=1}\int^{+\infty}_{-\infty}
\mathrm{Tr}\Big[ \pi^+_{\xi_n}(B_{1}(2)) \times\partial_{\xi_n} \sigma_{-2}(H_{X}^{-1})\Big](x_0)
\mathrm{d}\xi_n\sigma(\xi')\mathrm{d}x'\nonumber\\
=&-i\int_{|\xi'|=1}\int^{+\infty}_{-\infty}
\frac{i\xi_n  }{  2(\xi_n-i)^{5} (\xi_n+i)^{2}} \sum_{j,l=1}^{n-1}V_jW_l\xi_j\xi_lh'(0)
\mathrm{Tr}^{S(TM)}(\mathrm{id})\mathrm{d}\xi_n\sigma(\xi')\mathrm{d}x'\nonumber\\
&-i\int_{|\xi'|=1}\int^{+\infty}_{-\infty}
\frac{-i\xi_n  }{  2(\xi_n-i)^{5} (\xi_n+i)^{2}}  V_nW_n h'(0)
\mathrm{Tr}^{S(TM)}(\mathrm{id})\mathrm{d}\xi_n\sigma(\xi')\mathrm{d}x'\nonumber\\
=&-\frac{1}{4}\pi^{2}\sum_{j =1}^{n-1}V_jW_jh'(0)\mathrm{d}x'+\frac{3}{4}\pi^{2} V_nW_nh'(0)\mathrm{d}x'.
\end{align}
  Also, straightforward computations yield
  \begin{align}
B_{1}(3)=&-\sum_{j,l=1}^{n}V_jW_l\xi_j\xi_l\times\Big( -\sum_{j=1}^{n}\sqrt{-1}|\xi|^{-4} \big (X_{j}-\frac{1}{2}
\langle X ,  \partial_{j} \rangle   \big)\xi_j\Big)\nonumber\\
=&\Big(\sum_{j,l=1}^{n-1}V_jW_l\xi_j\xi_l+\sum_{j =1}^{n-1}V_jW_n\xi_j\xi_n+\sum_{ l=1}^{n-1}V_nW_l\xi_n\xi_l
+ V_nW_n\xi_n\xi_n\Big)\nonumber\\
&\times\sqrt{-1}|\xi|^{-4}\Big( \sum_{k=1}^{n-1} \big (X_{k}-\frac{1}{2}\langle X ,  \partial_{k} \rangle  \big)\xi_k
+  \big (X_{n}-\frac{1}{2} \langle X ,  \partial_{n} \rangle  \big)\xi_n\Big)\nonumber\\
=&\sqrt{-1}|\xi|^{-4}
\Big(\sum_{j,l=1}^{n-1}V_jW_l\xi_j\xi_l \big (X_{k}-\frac{1}{2}\langle X ,  \partial_{k} \rangle  \big)\xi_k
  +\sum_{j,l=1}^{n-1}V_jW_l\xi_j\xi_l \big (X_{n}-\frac{1}{2} \langle X ,  \partial_{n} \rangle  \big)\xi_n\nonumber\\
  &+\sum_{j,k =1}^{n-1}V_jW_n\xi_j\xi_n \big (X_{k}-\frac{1}{2}\langle X ,  \partial_{k} \rangle  \big)\xi_k
  +\sum_{j =1}^{n-1}V_jW_n\xi_j\xi_n\big (X_{n}-\frac{1}{2} \langle X ,  \partial_{n} \rangle  \big)\xi_n\nonumber\\
    &+\sum_{ k,l=1}^{n-1}V_nW_l\xi_n\xi_l\big (X_{k}-\frac{1}{2}\langle X ,  \partial_{k} \rangle  \big)\xi_k
  +\sum_{ l=1}^{n-1}V_nW_l\xi_n\xi_l\big (X_{n}-\frac{1}{2} \langle X ,  \partial_{n} \rangle  \big)\xi_n\nonumber\\
    &+V_nW_n\xi_n\xi_n\big (X_{k}-\frac{1}{2}\langle X ,  \partial_{k} \rangle  \big)\xi_k
  +V_nW_n\xi_n\xi_n\big (X_{n}-\frac{1}{2} \langle X ,  \partial_{n} \rangle  \big)\xi_n
\Big).
\end{align}
We note that $i<n,~\int_{|\xi'|=1}\xi_{i_{1}}\xi_{i_{2}}\cdots\xi_{i_{2d+1}}\sigma(\xi')=0$,
so we omit some items of $B_{1}(3)$ that have no contribution for computing {\bf case c)}.
By the Cauchy integral formula, we obtain
\begin{align}\label{36}
 &-i\int_{|\xi'|=1}\int^{+\infty}_{-\infty}
\mathrm{Tr}\Big[ \pi^+_{\xi_n}(B_{1}(3)) \times\partial_{\xi_n} \sigma_{-2}(H_{X}^{-1})\Big](x_0)
\mathrm{d}\xi_n\sigma(\xi')\mathrm{d}x'\nonumber\\
=&-i\int_{|\xi'|=1}\int^{+\infty}_{-\infty}
\frac{-\xi_n  }{  2(\xi_n-i)^{4} (\xi_n+i)^{2}}
\sum_{j,l=1}^{n-1}V_jW_l\xi_j\xi_l \big (X_{n}-\frac{1}{2} \langle X ,  \partial_{n} \rangle  \big)\xi_n
\mathrm{Tr}^{S(TM)}(\mathrm{id})\mathrm{d}\xi_n\sigma(\xi')\mathrm{d}x'\nonumber\\
&-i\int_{|\xi'|=1}\int^{+\infty}_{-\infty}
\frac{-\xi_n  }{  2(\xi_n-i)^{4} (\xi_n+i)^{2}}
\sum_{j,k =1}^{n-1}V_jW_n\xi_j\xi_n \big (X_{k}-\frac{1}{2}\langle X ,  \partial_{k} \rangle  \big)\xi_k
\mathrm{Tr}^{S(TM)}(\mathrm{id})\mathrm{d}\xi_n\sigma(\xi')\mathrm{d}x'\nonumber\\
&-i\int_{|\xi'|=1}\int^{+\infty}_{-\infty}
\frac{-\xi_n  }{  2(\xi_n-i)^{4} (\xi_n+i)^{2}}
\sum_{ k,l=1}^{n-1}V_nW_l\xi_n\xi_l\big (X_{k}-\frac{1}{2}\langle X ,  \partial_{k} \rangle  \big)\xi_k
\mathrm{Tr}^{S(TM)}(\mathrm{id})\mathrm{d}\xi_n\sigma(\xi')\mathrm{d}x'\nonumber\\
&-i\int_{|\xi'|=1}\int^{+\infty}_{-\infty}
\frac{-\xi_n  }{  2(\xi_n-i)^{4} (\xi_n+i)^{2}}
V_nW_n\xi_n\xi_n\big (X_{n}-\frac{1}{2} \langle X ,  \partial_{n} \rangle  \big)\xi_n
\mathrm{Tr}^{S(TM)}(\mathrm{id})\mathrm{d}\xi_n\sigma(\xi')\mathrm{d}x'\nonumber\\
=&-\frac{1}{3}\pi^{2}\sum_{j =1}^{n-1}V_jW_j\big (X_{n}-\frac{1}{2} \langle X ,  \partial_{n} \rangle  \big)     \mathrm{d}x'
 -\frac{1}{3}\pi^{2}\sum_{j =1}^{n-1}V_nW_j\big (X_{j}-\frac{1}{2} \langle X ,  \partial_{j} \rangle  \big)\mathrm{d}x' \nonumber\\
 &-\frac{1}{3}\pi^{2}\sum_{j =1}^{n-1}V_jW_n\big (X_{j}-\frac{1}{2} \langle X ,  \partial_{j} \rangle  \big)  \mathrm{d}x'
 -\pi^{2}V_nW_n \big (X_{n}-\frac{1}{2} \langle X ,  \partial_{n} \rangle  \big) \mathrm{d}x'.
 \end{align}

(ii) Explicit representation  of $B_{ 2}$, let $A(V)=\frac{1}{4}\sum_{i,j}\langle \nabla^{L}_{V}e_{i}, e_{j}  \rangle c(e_{i}) c(e_{j}) $,
  $A(W)=\frac{1}{4}\sum_{i,j}\langle \nabla^{L}_{W}e_{i}, e_{j}  \rangle c(e_{i})  c(e_{j}) $.
\begin{align}
B_{2}(x_0)=&\sigma_{1}\Big(\big(\nabla^{S(TM)}_{V}+\frac{3}{4}g(V,X)\big)
\big(\nabla^{S(TM)}_{W}+\frac{3}{4}g(W,X)\big) \Big)\sigma_{-2}(H_{X}^{-1})(x_0)|_{|\xi'|=1}\nonumber\\
=&\Big(\sum_{j,l=1}^n V_j\frac{\partial W_l }{\partial x_j }\sqrt{-1}\xi_l
+\sqrt{-1}\sum_{j =1}^nA(W)V_j\xi_j+\sum_{l =1}^n\sqrt{-1}A(V)W_l\xi_l\nonumber\\
&+\sum_{l =1}^n\frac{3}{4}g(V,X)W_l\sqrt{-1}\xi_l+\sum_{j =1}^n \frac{3}{4}g(W,X)V_j\sqrt{-1} \xi_j\Big)\times|\xi|^{-2}\nonumber\\
=&\Big(\sum_{j,l=1}^{n-1} V_j\frac{\partial W_l }{\partial x_j }\sqrt{-1}\xi_l
+\sqrt{-1}\sum_{j =1}^{n-1}A(W)V_j\xi_j+\sum_{l =1} ^{n-1}\sqrt{-1}A(V)W_l\xi_l\nonumber\\
&+\sum_{l =1}^{n-1}\frac{3}{4}g(V,X)W_l\sqrt{-1}\xi_l+\sum_{j =1}^{n-1} \frac{3}{4}g(W,X)V_j\sqrt{-1} \xi_j\Big)\times|\xi|^{-2}\nonumber\\
&+\Big(  V_n\frac{\partial W_n }{\partial x_n }\sqrt{-1}\xi_n
+\sqrt{-1}A(V)W_n\xi_n+ \sqrt{-1}A(W)V_n\xi_n\nonumber\\
&+ \frac{3}{4}g(V,X)W_n\sqrt{-1}\xi_n+  \frac{3}{4}g(W,X)V_n\sqrt{-1} \xi_n\Big)\times|\xi|^{-2}.
\end{align}
By the Cauchy integral formula, we obtain
\begin{align}
\pi^+_{\xi_n}(B_{2})(x_0)
=&\frac{-i}{2(\xi_n-i)}\Big(\sum_{j,l=1}^{n-1} V_j\frac{\partial W_l }{\partial x_j }\sqrt{-1}\xi_l
+\sqrt{-1}\sum_{j =1}^{n-1}A(W)V_j\xi_j+\sum_{l =1} ^{n-1}\sqrt{-1}A(V)W_l\xi_l\nonumber\\
&+\sum_{l =1}^{n-1}\frac{3}{4}g(V,X)W_l\sqrt{-1}\xi_l+\sum_{j =1}^{n-1} \frac{3}{4}g(W,X)V_j\sqrt{-1} \xi_j\Big) \nonumber\\
&+\frac{1}{2(\xi_n-i)}\Big(  V_n\frac{\partial W_n }{\partial x_n }\sqrt{-1}
+\sqrt{-1}A(V)W_n + \sqrt{-1}A(W)V_n \nonumber\\
&+ \frac{3}{4}g(V,X)W_n\sqrt{-1} +  \frac{3}{4}g(W,X)V_n\sqrt{-1}  \Big) \nonumber\\
=&\frac{-i}{2(\xi_n-i)}\Big(\sum_{j,l=1}^{n-1} V_j\frac{\partial W_l }{\partial x_j }\sqrt{-1}\xi_l
+\sqrt{-1}\sum_{j =1}^{n-1}A(W)V_j\xi_j+\sum_{l =1} ^{n-1}\sqrt{-1}A(V)W_l\xi_l\nonumber\\
&+\sum_{l =1}^{n-1}\frac{3}{4}g(V,X)W_l\sqrt{-1}\xi_l+\sum_{j =1}^{n-1} \frac{3}{4}g(W,X)V_j\sqrt{-1} \xi_j\Big) \nonumber\\
&+\frac{1}{2(\xi_n-i)}\Big( V_n\frac{\partial W_n }{\partial x_n }\sqrt{-1}
 + \frac{3}{4}g(V,X)W_n\sqrt{-1} +  \frac{3}{4}g(W,X)V_n\sqrt{-1}  \Big)\nonumber\\
 &+\frac{i}{8(\xi_n-i)}\Big(  \sum_{i,j}\langle \nabla^{L}_{V}e_{i}, e_{j}  \rangle c(e_{i})  c(e_{j})W_n
 + \sum_{i,j}\langle \nabla^{L}_{W}e_{i}, e_{j}  \rangle c(e_{i})  c(e_{j})V_n \Big).
\end{align}
Then we obtain
\begin{align}\label{36}
 &-i\int_{|\xi'|=1}\int^{+\infty}_{-\infty}
\mathrm{Tr}\Big[ \pi^+_{\xi_n}(B_{2}) \times\partial_{\xi_n} \sigma_{-2}(H_{X}^{-1})\Big](x_0)\mathrm{d}\xi_n\sigma(\xi')\mathrm{d}x'\nonumber\\
=&-i\int_{|\xi'|=1}\int^{+\infty}_{-\infty}
\frac{-i\xi_n  }{  (\xi_n-i)^{3} (\xi_n+i)^{2}} \Big( V_n\frac{\partial W_n }{\partial x_n }\sqrt{-1}
 + \frac{3}{4}g(V,X)W_n\sqrt{-1} +  \frac{3}{4}g(W,X)V_n\sqrt{-1}  \Big)\nonumber\\
&\times\mathrm{Tr}^{S(TM)}(\mathrm{id})\mathrm{d}\xi_n\sigma(\xi')\mathrm{d}x'\nonumber\\
&-i\int_{|\xi'|=1}\int^{+\infty}_{-\infty}
\frac{-i\xi_n  }{4(\xi_n-i)^{3} (\xi_n+i)^{2}} \mathrm{Tr}\Big(  \sum_{i,j}\langle \nabla^{L}_{V}e_{i}, e_{j}  \rangle c(e_{i})  c(e_{j})W_n
  \Big)
\mathrm{Tr}^{S(TM)}(\mathrm{id})\mathrm{d}\xi_n\sigma(\xi')\mathrm{d}x'\nonumber\\
&-i\int_{|\xi'|=1}\int^{+\infty}_{-\infty}
\frac{-i\xi_n  }{4(\xi_n-i)^{3}} \mathrm{Tr}\Big(   \sum_{i,j}\langle \nabla^{L}_{W}e_{i}, e_{j}  \rangle c(e_{i})  c(e_{j})V_n \Big)
\mathrm{Tr}^{S(TM)}(\mathrm{id})\mathrm{d}\xi_n\sigma(\xi')\mathrm{d}x' \nonumber\\
=& 2i\pi^{2}\Big( iV_n\frac{\partial W_n }{\partial x_n } + \frac{3}{4}\sqrt{-1}g(V,X)W_n +\frac{3}{4}\sqrt{-1}g(W,X)V_n  \Big) \mathrm{d}x'
+\frac{1}{2}\pi^{2}V_{n} \sum_{ j}\langle \nabla^{L}_{W}e_{j}, e_{j}  \rangle \mathrm{d}x'\nonumber\\
&+\frac{1}{2}\pi^{2}W_{n} \sum_{ j}\langle \nabla^{L}_{V}e_{j}, e_{j}  \rangle \mathrm{d}x'.
\end{align}

(iii) Explicit representation  of $B_{3}$,
\begin{align}
B_{3}(x_0)=&\sum_{j=1}^{n}\sum_{\alpha}\frac{1}{\alpha!}\partial^{\alpha}_{\xi}\Big[\sigma_{2}\Big(\big(\nabla^{S(TM)}_{V}+\frac{3}{4}g(V,X)\big)
\big(\nabla^{S(TM)}_{W}+\frac{3}{4}g(W,X)\big) \Big) \Big]
D_x^{\alpha}\big[\sigma_{-2}(H_{X}^{-1})\big](x_0)|_{|\xi'|=1}\nonumber\\
=&\sum_{j=1}^{n}\partial_{\xi_{j}}\Big[\sigma_{2}\Big(\big(\nabla^{S(TM)}_{V}+\frac{3}{4}g(V,X)\big)
\big(\nabla^{S(TM)}_{W}+\frac{3}{4}g(W,X)\big) \Big)\Big]
(-\sqrt{-1})\partial_{x_{j}}\big[\sigma_{-2}(H_{X}^{-1})\big]\nonumber\\
=&\sum_{j=1}^{n}\partial_{\xi_{j}}\big[-\sum_{j,l=1}^nV_jW_l\xi_j\xi_l\big]
(-\sqrt{-1})\partial_{x_{j}}\big[|\xi|^{-2}\big]\nonumber\\
=&\Big[\sum_{j=1}^{n-1} V_jW_{n}\xi_{j}
+\sum_{l=1}^{n}V_nW_{l}\xi_{l}
+2V_nW_{n}\xi_{n}\Big]\frac{-\sqrt{-1}h'(0)}{|\xi|^{4}} .
\end{align}
  Also, straightforward computations yield
\begin{align}\label{36}
 &-i\int_{|\xi'|=1}\int^{+\infty}_{-\infty}
\mathrm{Tr}\Big[ \pi^+_{\xi_n}(B_{3}) \times\partial_{\xi_n} \sigma_{-2}(H_{X}^{-1})\Big](x_0)\mathrm{d}\xi_n\sigma(\xi')\mathrm{d}x'\nonumber\\
=&-i\int_{|\xi'|=1}\int^{+\infty}_{-\infty}
\frac{-i\xi_n(2+i\xi_n) }{ 2(\xi_n-i)^{4} (\xi_n+i)^{2}}h'(0)\Big(\sum_{j =1}^{n}V_{j}W_{n} \xi_{j}+\sum_{l =1}^{n}V_{n}W_{l} \xi_{l}\Big)
\mathrm{Tr}^{S(TM)}(\mathrm{id})\mathrm{d}\xi_n\sigma(\xi')\mathrm{d}x'\nonumber\\
&-i\int_{|\xi'|=1}\int^{+\infty}_{-\infty}
\frac{ \xi_n  }{  (\xi_n-i)^{4} (\xi_n+i)^{2}}h'(0)V_{n}W_{n}
\mathrm{Tr}^{S(TM)}(\mathrm{id})\mathrm{d}\xi_n\sigma(\xi')\mathrm{d}x'\nonumber\\
=&2\pi^{2}V_{n}W_{n}h'(0)\mathrm{d}x'.
\end{align}
Summing up (3.52),(3.54),(3.57),(3.59) leads to the result of $\Phi_5$.

Let $X=X^T+X_n\partial_n,~Y=Y^T+Y_n\partial_n,$ then we have $\sum_{j=1}^{n-1}X_jY_j(x_{0})=g(X^T,Y^T).$
Now $\Phi$ is the sum of the cases (a), (b) and (c). Therefore, we get
\begin{align}\label{795}
\Phi=&\sum_{i=1}^5\Phi_i
= \Big[-\frac{1}{3} \pi^{2}\partial_{ x_n} (g(V^T,W^T) ) -2\pi^{2}\partial_{ x_n} (V_{n}W_{n})
-\frac{7+5i}{12}\pi^{2}h'(0)g(V^T,W^T)\nonumber\\
&- \pi^{2}\Big(\sum_{j =1}^{n-1} \big (X_{j}-\frac{1}{2}\langle X ,  \partial_{j} \rangle   \big)V_{j}W_{n}+
               \sum_{l =1}^{n-1} \big (X_{l}-\frac{1}{2}\langle X ,  \partial_{j} \rangle   \big)V_{n}W_{l}\Big)\nonumber\\
&+2i\pi^{2}\Big( iV_n\frac{\partial W_n }{\partial x_n } + \frac{3}{4}\sqrt{-1}g(V,X)W_n +\frac{3}{4}\sqrt{-1}g(W,X)V_n  \Big) \nonumber\\
&+\frac{11+20i}{4}\pi^{2}V_{n}W_{n}h'(0)
-(1+i)\pi^{2}V_{n}W_{n} \big (X_{n}-\frac{1}{2}\langle X ,  \partial_{n} \rangle   \big)
\nonumber\\
&+\frac{1}{2}\pi^{2}V_{n} \sum_{ j}\langle \nabla^{L}_{W}e_{j}, e_{j}  \rangle
+\frac{1}{2}\pi^{2}W_{n} \sum_{ j}\langle \nabla^{L}_{V}e_{j}, e_{j}  \rangle \Big]\mathrm{d}x'\nonumber\\
=& \Big[-\frac{1}{3} \pi^{2}\partial_{ x_n} (g(V^T,W^T) )
-\frac{7+5i}{12}\pi^{2}h'(0)g(V^T,W^T)\nonumber\\
&-2\pi^{2}\partial_{ x_n} (V_{n}W_{n})- \pi^{2}\Big(g(X^T,V^T)W_{n}+g(X^T,W^T)V_{n}\Big)\nonumber\\
&+2i\pi^{2}\Big( iV_n\frac{\partial W_n }{\partial x_n } + \frac{3}{4}\sqrt{-1}g(V,X)W_n +\frac{3}{4}\sqrt{-1}g(W,X)V_n  \Big) \nonumber\\
&+\frac{11+20i}{4}\pi^{2}V_{n}W_{n}h'(0)
-(1+i)\pi^{2}V_{n}W_{n} X_{n} \Big]\mathrm{d}x',
\end{align}
where in the lase equality we have used
$\langle \nabla^{L}_{V}e_{j}, e_{j}  \rangle+\langle e_{j}, \nabla^{L}_{V}e_{j}  \rangle=V(e_{j},e_{j})=0$,
, $\langle \nabla^{L}_{V}e_{j}, e_{j}  \rangle=0 $, and
\begin{align}
\sum_{j =1}^{n-1} \big (X_{j}-\frac{1}{2}\langle X ,  \partial_{j} \rangle   \big)V_{j}(x_{0})
=&\sum_{j =1}^{n-1} \big (X_{j}-\frac{1}{2}\langle X , e_{j} \rangle   \big)V_{j}(x_{0})
=\sum_{j =1}^{n-1} \big (X_{j}-\frac{1}{2}X_{j}   \big)V_{j}(x_{0})\nonumber\\
=&\frac{1}{2}\sum_{j =1}^{n-1} X_{j}V_{j}=\frac{1}{2}g(X^{T},V^{T}).
\end{align}
Then we obtain
\begin{thm}\label{thmb1}
 Let $M$ be a 4-dimensional compact spin manifold with boundary $\partial M$,
 the spectral Einstein functionals for the equivariant Bismut Laplacian $H_{X}$ as follows:
\begin{align}
&\widetilde{Wres}\Big[\pi^{+} \big(\nabla^{S(TM)}_{V}
+\frac{3}{4}g(V,X)\big)\big(\nabla^{S(TM)}_{W}+\frac{3}{4}g(W,X)\big)H_{X} ^{-1}
\circ\pi^{+}H_{X}^{-1}\Big]\nonumber\\
=&\frac{4\pi^{2}}{3} \int_{M}\big(Ric(V,W)-\frac{1}{2}sg(V,W)\big) vol_{g}
+3\pi^{2}\int_{M} \Big(g\big(\nabla_{V}^{TM}X,W\big)-g\big(\nabla_{W}^{TM}X,V\big) \Big)vol_{g}\nonumber\\
&+ \int_{M}\Big(\frac{1}{2}s^{g}+ \mathrm{div}^{g}(X)+ |X|^{2}
-2Tr\big( \mu^{S(TM)}(X)\big) \Big)g(V,W) vol_{g} \nonumber\\
&+\int_{\partial M}\Big[-\frac{1}{3} \pi^{2}\partial_{ x_n} (g(V^T,W^T) )
-\frac{7+5i}{12}\pi^{2}h'(0)g(V^T,W^T)\nonumber\\
&-2\pi^{2}\partial_{ x_n} (V_{n}W_{n})- \pi^{2}\Big(g(X^T,V^T)W_{n}+g(X^T,W^T)V_{n}\Big)\nonumber\\
&-2\pi^{2}\Big(  V_n\frac{\partial W_n }{\partial x_n } + \frac{3}{4} g(V,X)W_n +\frac{3}{4} g(W,X)V_n  \Big) \nonumber\\
&+\frac{11+20i}{4}\pi^{2}V_{n}W_{n}h'(0)
-(1+i)\pi^{2}V_{n}W_{n} X_{n} \Big]\mathrm{d}x'.
\end{align}
where $\mathrm{div}^{g}(X)$ denotes the divergence of $X$, and $s$ denotes the scalar curvature.
\end{thm}

\subsection{Spectral Einstein functionals of the Bismut Laplacian for three dimensional manifolds with boundary}

In this section, we compute the noncommutative residue
 for the equivariant Bismut Laplacian $H_{X}$
  on $3$-dimensional oriented compact manifolds $M$ with boundary $\partial M$.
  As in \cite{Wa1}, for an odd-dimensional manifolds with boundary, we have the formula
\begin{align}
\widetilde{Wres}\Big[\pi^{+} \big(\nabla^{S(TM)}_{V}
+\frac{3}{4}g(V,X)\big)\big(\nabla^{S(TM)}_{W}+\frac{3}{4}g(W,X)\big)H_{X} ^{-1}
\circ\pi^{+}H_{X}^{-1}\Big] =\int_{\partial M}\widetilde{\Phi}.
\end{align}
By (3.9),  when $r+l-k-|\alpha|-j-1=-3$,
so we get ~$r=0,~\ell=-2,~k=j=|\alpha|=0$, then
\begin{align}\label{36}
\widetilde{\Phi} =&-i\int_{|\xi'|=1}\int^{+\infty}_{-\infty}
\mathrm{Tr}\Big[ \pi^+_{\xi_n}
\sigma_{0}\Big(\big(\nabla^{S(TM)}_{V}+\frac{3}{4}g(V,X)\big)\big(\nabla^{S(TM)}_{W}+\frac{3}{4}g(W,X)\big)H_{X}^{-1}\Big) \nonumber\\
&\times\partial_{\xi_n} \sigma_{-2}(H_{X}^{-1})\Big](x_0)\mathrm{d}\xi_n\sigma(\xi')\mathrm{d}x'.
\end{align}
By Lemma 3.4, we have
\begin{align}\label{62}
\partial_{\xi_n}\sigma_{-2}(H_{X}^{-1})(x_0)|_{|\xi'|=1}=-\frac{2\xi_n}{(\xi_n^2+1)^2}.
\end{align}
On the other hand, by the Cauchy integral formula, we obtain
\begin{align}\label{38}
&\pi^+_{\xi_n}\sigma_{0}\Big(\big(\nabla^{S(TM)}_{V}+\frac{3}{4}g(V,X)\big)\big(\nabla^{S(TM)}_{W}
+\frac{3}{4}g(W,X)\big)H_{X}^{-1}\Big)(x_0)|_{|\xi'|=1}\nonumber\\
=&\pi^+_{\xi_n}\Big[\sigma_{2}\Big(\big(\nabla^{S(TM)}_{V}+\frac{3}{4}g(V,X)\big)\big(\nabla^{S(TM)}_{W}
+\frac{3}{4}g(W,X)\big)\Big)\sigma_{-2}(H_{X}^{-1})\Big](x_0)|_{|\xi'|=1}\nonumber\\
=&\frac{i}{2(\xi_n-i)}\sum_{j,l=1}^{n-1}V_jW_l\xi_j\xi_l+\frac{1}{2(\xi_n-i)}V_nW_n-\frac{1}{2(\xi_n-i)}\sum_{j=1}^{n-1}V_jW_n\xi_j\nonumber\\
&-\frac{1}{2(\xi_n-i)}\sum_{l=1}^{n-1}V_nW_l\xi_l.
\end{align}
We note that $i<n,~\int_{|\xi'|=1}\xi_{i_{1}}\xi_{i_{2}}\cdots\xi_{i_{2d+1}}\sigma(\xi')=0$,
so we omit some items that have no contribution for computing $\widetilde{\Phi}$.
Then we obtain
\begin{align}\label{36}
\widetilde{\Phi} =&-i\int_{|\xi'|=1}\int^{+\infty}_{-\infty}
\frac{-i\xi_n  }{   (\xi_n-i)^{3} (\xi_n+i)^{2}} \sum_{j,l=1}^{n-1}V_jW_l\xi_j\xi_lh'(0)
\mathrm{Tr}^{S(TM)}(\mathrm{id})\mathrm{d}\xi_n\sigma(\xi')\mathrm{d}x'\nonumber\\
&-i\int_{|\xi'|=1}\int^{+\infty}_{-\infty}
\frac{ i\xi_n  }{   (\xi_n-i)^{3} (\xi_n+i)^{2}}  V_nW_n h'(0)
\mathrm{Tr}^{S(TM)}(\mathrm{id})\mathrm{d}\xi_n\sigma(\xi')\mathrm{d}x'\nonumber\\
=& \Big(\frac{i\pi}{4}g(V^T,W^T) -\frac{i}{2}\pi V_nW_n \Big)\mathrm{d}x'.
\end{align}
Therefore, we get the following theorem
\begin{thm}\label{thmb1}
 Let $M$ be a 3-dimensional compact spin manifold with boundary $\partial M$ ,
  then we get spectral   functionals for the equivariant Bismut Laplacian $H_{X}$,
\begin{align}
&\widetilde{Wres}\Big[\pi^{+} \big(\nabla^{S(TM)}_{V}
+\frac{3}{4}g(V,X)\big)\big(\nabla^{S(TM)}_{W}+\frac{3}{4}g(W,X)\big)H_{X} ^{-1}
\circ\pi^{+}H_{X}^{-1}\Big] \nonumber\\
=& \int_{\partial M}  \Big(\frac{i\pi}{4}g(V^T,W^T) -\frac{i}{2}\pi V_nW_n \Big)\mathrm{d}x'.
\end{align}
\end{thm}

\section*{ Acknowledgements}
The first author was supported by NSFC. 11501414. The second author was supported by NSFC. 11771070.
 The authors also thank the referee for his (or her) careful reading and helpful comments.

\end{document}